\documentclass[preprint,10pt]{elsarticle}
\usepackage[paperwidth=8.5in,paperheight=11in,top=1.00in, bottom=1.00in, left=1.00in, right=1.00in]{geometry}
\usepackage{babel}
\usepackage{float}
\usepackage[utf8]{inputenc}
\usepackage[T1]{fontenc}
\usepackage{hyperref}
\usepackage{xcolor}
\usepackage{fancyhdr}
\usepackage{indentfirst}
\usepackage{graphicx}
\usepackage{newlfont}
\usepackage{amssymb}
\usepackage{amsmath}
\usepackage{mathtools}
\usepackage{latexsym}
\usepackage{amsthm}
\usepackage{amsthm,thmtools}
\usepackage{mathdots}
\usepackage{dsfont}
\DeclareMathAlphabet\mathrsfso      {U}{rsfso}{m}{n}	

\newtheorem{remark}{Remark}
\definecolor{blue}{RGB}{0, 0, 255}

\newcommand{\T}{\bar{T}} 
\newcommand{\K}{\mathbf{k}} 
\newcommand{\PB}{P^{\text{B}}} 
\newcommand{\PC}{P^{\text{C}}} 
\newcommand{\aB}{\alpha^{\text{B}}} 
\newcommand{\aC}{\alpha^{\text{C}}} 
\newcommand{\gB}{g^B} 
\newcommand{\gC}{g^C} 
\newcommand{\V}{\hat{V}} 
\newcommand{\CVA}{\text{CVA}} 
\newcommand{\RE}{\text{RE}} 
\newcommand{\Eq}{\text{E}} 
\newcommand{\EAD}{\text{EAD}} 
\newcommand{\PFE}{\text{PFE}} 
\newcommand{\RC}{\text{RC}} 
\newcommand{\AddOn}{\text{AddOn}} 
\newcommand{\MF}{\text{MF}} 
\newcommand{\SF}{\text{SF}} 
\newcommand{\RWA}{\text{RWA}} 
\newcommand{\REPO}{\text{Repo}} 

\newcommand{\A}{\mathcal{A}} 
\newcommand{\F}{\mathrsfso{F}} 
\newcommand{\R}{\mathbb{R}} 
\newcommand{\prob}{\mathbb{P}} 
\newcommand{\Q}{\mathbb{Q}} 
\newcommand{\N}{\mathbb{N}} 
\newcommand{\E}{\mathbb{E}} 
\newcommand{\Leb}{\mathbb{L}} 
\newcommand{\ip}[2]{\langle #1, #2\rangle}
\DeclareMathOperator*{\esssup}{ess\,sup}

\newcommand{\Smax}{\bar{S}} 

\newcommand{\Vnum}{\hat{v}} 
\newcommand{\C}{\mathcal{C}}
\newcommand{\Hv}{\mathcal{H}} 
\newcommand{\D}{\mathcal{D}}
\newcommand{\Kv}{\mathcal{K}} 
\usepackage{enumitem}

\newtheorem{prop}{Proposition}
\newtheorem{cor}{Corollary}
\newtheorem{teo}{Theorem}

\newtheorem{definition}{Definition}

\journal{Applied Mathematics and Computation}

\newcommand{\firstReviewer}[1]{{\color{black}{#1}}}
\newcommand{\secondReviewer}[1]{{\color{black}{#1}}}

\begin{document}
\begin{frontmatter}
\title{Mathematical models and numerical methods for \\ a 
 capital valuation adjustment (KVA) problem}
\author[auth1]{D. Trevisani}\ead{davide.trevisani@udc.es}
\address[auth1]{Department of Mathematics, Faculty of Informatics and CITIC, Campus Elvi\~na s/n, 15071-A Coru\~na (Spain)}

\author[auth1]{J. G. L\'opez-Salas}\ead{jose.lsalas@udc.es}

\author[auth1]{C. V\'azquez}\ead{carlosv@udc.es}
 
\author[auth1]{J. A. Garc\'ia-Rodr\'iguez}\ead{jose.garcia.rodriguez@udc.es}

\begin{abstract}
In this work we rigorously establish mathematical models to obtain the capital valuation adjustment (KVA) as part of the total valuation adjustments (XVAs). For this purpose, we use a semi-replication strategy based on market theory. We formulate single-factor models in terms of expectations and PDEs. For PDEs formulation, we rigorously obtain the existence and uniqueness of the solution, as well as some regularity and qualitative properties of the solution. Moreover, appropriate numerical methods are proposed for solving the corresponding PDEs. Finally, some examples show the numerical results for call and put European options and the corresponding XVA that includes the KVA.
\end{abstract}

\end{frontmatter}

\section{Introduction}
Capital requirements, also known as regulatory capital, capital adequacy, or capital base is the amount of capital a bank or other financial institutions are required to hold by its financial regulators. 
It is usually defined as a percentage of risk-weighted-assets, and thus depends on the risks attached to the portfolio of a financial institution; the higher the risks the higher the measure (e.g. formulas contained in \cite{crebasel,marbasel}). 
Throughout history, banks have autonomously set aside capital to be used in periods of crisis.  However, with the introduction of Basel I in 1988, a group of nations officially organized themselves and established regulations. Subsequently, the world economic crisis in 2008 underscored the inadequacy of Basel II in force at that time, and this ultimately resulted in the introduction of Basel III in 2011 (see \cite[Section 12.1]{green2015xva} for a historical evolution of Basel framework).
However, while making banks more resilient, these measures drew their attention due to the increase in operating costs caused by the capital requirements themselves. Shareholders always ask for higher returns because of the higher risk they bear. While insurance regulations exist for risk margins (the counterpart of KVA in Solvency II \cite{solvency2004risk}), this problem is not treated in Basel III.
Moreover, contrary to what happens with debt holders, where the return they receive depends solely on market conditions and can easily be extracted from bond quotes, the return expected by the shareholders is unknown. Further idiosyncrasies might also occur because each institution might opt for different methodologies to calculate the requirements, and ultimately, capital management depends on the single institution. As a consequence, currently, the financial industry might agree on existence, rather than on a definition of KVA. 

Even though all research about KVA agrees in considering a new cost yielding at a certain hurdle rate, for the aforementioned reasons, actually we cannot identify a unique stream of research around this topic.
For example, the approach initiated in \cite{albanese2016capital} is inspired by regulations used in insurance (see \cite[Solvency II]{solvency2004risk}).  In \cite{albanese2016capital} the authors define the KVA in terms of a forward-backward stochastic differential equation (FBSDE), which is obtained by selecting an optimal economic capital policy.
Moreover, only market-risk capital is considered and there is also the idea to treat the KVA as a retained earning. This idea was also proposed in \cite{garcia2019retained}, where a PDE model is obtained by extending the semi-replication arguments of \cite{burgard2013funding} (see also \cite{burgard2011balance,burgard2011partial,piterbarg2010funding}). A similar semi-replication argument can then be found in \cite{green2014kva}, where the KVA is no more a retained earning and so the entire cost of capital is charged to the client.

The present work mainly follows the approach in \cite{green2014kva}, although we also present a one-dimensional version of the PDE model of \cite{garcia2019retained}. Our aims in this work are multiple. In the first place, we want to frame semi-replication for KVA (and XVAs in general) under a solid and classical market theory. As highlighted in \cite{brigo2012illustrating}, the approach proposed in \cite{piterbarg2010funding} (and by the extension also semi-replication) contains mathematical imprecisions which can be overcome with the theory of markets with dividends (see \cite{duffie}, for example). Accordingly, we apply this theory to construct the KVA model. We will show that under this approach some remarks about the lack of arbitrage can be made (see \ref{sec:lackofarb}). We therefore devote ourselves to studying a single-factor model with a simplified definition of regulatory capital. In particular, we define capital requirements through SACCR and the basic approach for CVA capital. We also neglect completely market risk capital (thus FRTB-capital), under the assumption of a perfectly hedged portfolio.
Once the PDE model is thoroughly stated, we conduct a mathematical analysis for this model by proving well-posedness in a mild sense of the PDE, and then we deduce some regularity results. We finally propose suitable numerical methods to solve the pricing problem for European vanilla options. Namely, we mainly discuss an application of the IMEX-LDG scheme proposed in \cite{wang2015stability,wang2016stability}. Numerical solutions of a system of FBSDEs are also considered using the method in \cite{gltc2016}. We finally show and discuss the obtained results.

The article is organized as follows. In Section \ref{sec:matheModelling} we establish the mathematical model by semi-replication arguments. Section \ref{sec:PDE and BSDE} is devoted to the mathematical analysis of the PDE model to establish the existence and uniqueness of the solution, as well as its regularity. In Section \ref{sec:numMethods} we describe the proposed numerical method and Section \ref{sec:numResults} shows some numerical examples of call and put options pricing. Finally, several Appendices complete the article.

\section{Mathematical modelling} \label{sec:matheModelling}
In this part, we present a mathematical framework for European options pricing with XVA in a timeline $[0,T]$. By an economy we mean a triple made of:
\begin{itemize}
    \item [i)] A probability space $(\Omega,\prob, \mathcal{F}, \{\F\}_{t\in [0,T]})$ with usual assumptions on the filtration, with $\prob$ being the real-world measure.
    \item [ii)] A couple $(Y, D)$ of stochastic adapted processes, where $Y$ denotes the prices of the assets in the economy, and $D$ is the so-called cumulative dividend process. Specifically, $D$ represents additional cashflows caused by holding the assets, so that it can track gains and costs that are not caused directly by trading.
    \item [iii)] A set of admissible trading strategies.
\end{itemize}
In the economy that we consider, the process $Y$ is driven by three sources of risk:
\begin{itemize}
    \item A one-dimensional Brownian motion $W$ driving the underlying stock $S$.
    \item A single-jump process $J^B$ that jumps when the firm $B$ defaults.
    \item A single-jumps process $J^C$ that jumps when the counterparty $C$ defaults.
\end{itemize} 
All other market factors, such as interest rates, intensities of default, and the capital hurdle rate are here assumed to be deterministic. However, the XVA measures here presented can be defined as expectations under an equivalent martingale measure as in \cite{green2015xva}. Under the standard assumption that an equivalent martingale measure for the economy exists  (see \cite{brigo2001interest}), a multi-dimensional PDE model can be then deduced through the Feynman-Kac representation formula.

Analogously to \cite{arregui2017pde} or \cite{simonella2023xva}, for example, we deploy a semi-replication argument that shows that the XVA-adjusted option price is the solution to a linear or a semilinear PDE, depending on the choice of Mark-to-Market (MTM) price of the derivative at default. Concerning the notation used here and in the following, we address the reader to the tables in \ref{sec:notation}.
\subsection{Asset-dividend dynamics} 
\label{subsec:asset-dividend}
In our economy, the price of the assets is a multi-dimensional stochastic process $Y$  on $[0, T]$. More precisely, let $$Y_t=(\V_t,X_t,CA_t,(\REPO^S)_t,(\REPO^C)_t,\PB_{t,\T},B_{t,\T}), \quad t\in [0, T],$$ where the components of $Y_t$ are defined as follows:
\begin{equation}
\label{eq:assets}
\begin{aligned}
& \V_t & \text{XVA-adjusted derivative}\\
& X_t &\text{Collateral-account}\\
& CA_t & \text{Capital account}\\
&(\REPO^S)_t=0 & \text{Stock-\secondReviewer{Repo}}\\
& (\REPO^C)_t=0 & \text{Counterparty-bond-\secondReviewer{Repo}}\\
& d\PB_{t,\T}/\PB_{t,\T}=r_t^Bdt- (1-R^B_t)dJ^B_t & \text{Firm's own bond}\\
& dB_{t,\T}/B_{t,\T}=r_t \,dt & \text{Riskless account}.
\end{aligned}
\end{equation} 
At time $t=0$ the $\sigma$-algebra is degenerate, thereby the initial conditions $Y_0$ of the above dynamics are not random.
This section aims to characterize the unknown dynamics of $\V$ so that the economy does not contain arbitrages. 
We denote the cumulative dividend process, i.e. additional cashflows caused by the assets $Y$ as $$D=(D^{\V},D^{X}, D^{CA},D^{\REPO^S},D^{\REPO^C}D^{P\REPO^B},D^{B}).$$ The initial value of the cumulative dividend process is $D_0=0$, while its evolution is described in the following points:
\begin{itemize}
\item  $\REPO^S$ represents a repurchase agreement on a single dividend-paying-stock $S$. By modelling the stock price $S$ as a geometric Brownian motion
 $$dS_t/S_t=\mu dt+\sigma dW_t,$$
the cashflow caused by the $\REPO$ is described as
 $$dD^{\REPO^S}_t=dS_t+ (\gamma_S-q_S)S_tdt.$$
Specifically, in a short time, the stock is sold and re-bought, yielding a cashflow of $dS_t$. Furthermore, on the one hand, the holder of the stock is remunerated with a stock-dividend amount $\gamma_S S_t\, dt$, while on the other hand, a \secondReviewer{Repo} rate of $q_S$ is paid as agreed. 
\item Similarly, $\REPO^C$ represents a repurchase agreement on a single counterparty bond. So, we have
 $$dD^{\REPO_C}_t=d\PC_t-q_C\PC_t\, dt,$$
where
$$d\PC_t/\PC_t=r^C_t dt- (1-R^C_t)dJ^C_t$$ 
models the dynamics of a defaultable counterparty bond. In particular, when $J^C$ jumps, the bond defaults, and its value is reduced to the level given by its recovery rate $R^C_t$,  which is generally time-dependent. An analogous consideration holds for the Firm's bond $P^B$ as indicated in \eqref{eq:assets}.
\item $D^{B}$ and $D^{P^B}$ are constant and equal to zero. The riskless account has no additional costs, i.e., it is dividend-zero. Being positive, it can then be chosen as a numeraire (see 2.2 of \cite{brigo2001interest}).
\item The collateral account $X$ is an asset that creates a cashflow due to collateral management and margining procedures. In this respect, additional details can be found for example in \cite{pallavicini2012funding} with explicit cashflow formulas in discrete and continuous time. In this model,
$$dD^{X}_t=r^X_t X_t \, dt-dX_t,$$
that is, margins $dX_t$ and interests payments $r^X_t X_t\, dt$ are settled in continuous time. It is worth stressing that this flow is not caused by trading the collateral account, so it applies whenever a portfolio contains $X$. 
\item The capital account $CA$ represents the regulatory capital amount available at the derivative desk for funding purposes. In this part, we follow the model in \cite{green2014kva}, for which the cumulative dividend process satisfies 
$$dD^{CA}_t=\gamma^\K \K_t \, dt- d(CA)_t.$$
In particular, the amount $\gamma^\K \K \, dt$ represents the remuneration to shareholders for holding the whole regulatory capital $\K$. For simplicity, this remuneration is modeled as a continuous flux of dividends, like the stock. Furthermore, the capital account is
$$(CA)_t=\phi \K_t \qquad \phi\in [0,1],$$
i.e., it is equal to a constant ratio of the capital requirement. Other KVA pricing models, such as the one presented in \cite{garcia2016recursive}, differ from \cite{green2014kva} in the choice of $CA$ and $D^{CA}$. 
\item Finally, we associate the derivative $\V$ itself with a dividend process
\begin{equation}
\label{eq:def hedg err}
dD^{\V}=-\epsilon^H_t dJ^{B}_t \quad \text{with} \quad \epsilon^H_t:=\Delta_B \V+(1-R^B_t)(\V_t-X_t-(CA)_t).
\end{equation}
Here $\Delta_B \V$ represents the variation of price of $\V$ in case of $B$ defaults (see Subsection \ref{subsec:close-out}).
The term $\epsilon_H$ is known in \cite{burgard2013funding} as ``hedging error''. This term is related to the fact that, in practice, hedging own default is not possible. As a consequence, using risky bonds for funding purposes will always generate an imperfect hedging of a derivative. The hedging error simply quantifies this imperfection (and from here, the term ``semi-replication''). By imposing $dD^{\V}=-\epsilon^H_t dJ^{B}_t$ we are assuming that this excessive gain is paid by the firm $B$ in case of default, as it is likely to happen when bankruptcy costs are considered.
\end{itemize}
Before we start the semi-replication procedure, we define the gain process $G:=Y+D$ along with the discounted gain process described by
$$dG^B_t=d\left(\frac{Y_t}{B_t}\right)+\frac{dD_t}{B_t}.$$
We then assume a zero basis for bond-CDS and bond-\secondReviewer{Repo}, that is
\begin{equation}
\label{eq:spread-intensity}
\lambda^C_t=\frac{r^C_t-q^C_t}{1-R^C_t}, \quad
\lambda^B_t=\frac{r^B_t-r_t}{1-R^B_t},
\end{equation}
where $\lambda^C$ and $\lambda^B$ denote the jump-intensity of $J^C$ and $J^B$, respectively (i.e., the intensity of default of $C$ and $B$, respectively). We highlight that this is equivalent to assuming a zero price of default risk for both $B$ and $C$. Indeed, the assets $\REPO^C$ and $P^B$ are in the gain process by the equations 
$$d(\REPO^C)_t+ dD^{\REPO^C}_t=(r^C_t-q^C_t)\PC\,dt-(1-R^C_t)\PC dJ^C_t,$$
$$d\PB_t+ dD^{\PB}_t=r^B_t\PB\,dt-(1-R^B_t)\PB dJ^B_t.$$
Therefore, under this assumption, the same components in the discounted gain process have dynamics
$$(r^C_t-q^C_t)\frac{\PC_t}{B_t}\,dt-(1-R^C_t)\frac{\PC_t}{B_t} dJ^C_t,\qquad (r^B_t-r_t)\frac{\PB_t}{B_t}\,dt-(1-R^B_t)\frac{\PB_t}{B_t} dJ^B_t,$$
and so they describe two $\prob$-martingales.
\subsection{Pricing by semi-replication}
\label{subsec:semi-repl}
In this part, we deduce the pricing model of $\V$ for the economy defined above. In particular, we will take into account market and default risks, along with funding, collateral, and regulatory capital costs. We will set a portfolio that perfectly replicates $\V$ allowing us to deduce a PDE for pricing purposes. In order to apply our argument we need to assume the function $\V=\V(t,S, J^C, J^B)$ to be regular enough so that Itô Lemma for jump-diffusion processes can be applied.
We then define a self-financing hedging portfolio $\theta$ containing $\V$. The components in $\theta$ are given by
\begin{equation}
\label{eq:port-notation}
\theta:=(1,-1,-1, \delta, \aC, \alpha^B,0).
\end{equation}
In particular, the positions in $\V$, the collateral and capital account are $1,-1$, and $-1$ respectively. This is a requirement that makes the portfolio admissible (see \ref{sec:lackofarb}). Specifically, the risky derivative $\V$ is attached with a single Credit Support Annex (CSA), and a regulatory capital amount. As a result, there is only one collateral account and one capital requirement for every derivative $\V$. Furthermore, the $-1$ for the two accounts means that these latter are used to finance the derivative. In particular, an amount $\V_0-X_0-\phi \K_0+\alpha^BP^B$ is invested at time $0$.
As the portfolio is self-financing, by \eqref{eq:assets} the value $\Pi_t$ of the strategy is
\begin{equation}
 \Pi_t=\theta_t\cdot Y_t=\V_t-X_t-\phi \K_t+\aB_t\PB_t.
\end{equation}   
We set $\Pi_t=0$ so that the \secondReviewer{Repo} positions $\delta$ and $\aC_t$ can replicate $\V$. In particular, the relation
\begin{equation}
\label{eq:financing constr K}
\aB_t\PB_t=X_t+\varphi \K_t-\V_t,
\end{equation}
connects the funding position with the value of the collateral, the regulatory capital available to the desk, and the derivative price. Namely, as long as the right-hand side is positive the bank has a long position in the account. On the other hand, a negative right-hand side implies a negative cashflow on the funding account position because of the funding rate attached to the account. This last statement is satisfied if and only if $\V_t> X_t+\varphi \K_t$, i.e., when the value of $\V_t$ exceeds the collateral given the counterparts plus the capital for funding purposes.
In order to simplify the notation, in the following, we will suppress the subscripted time variable $t$. 
For a market having cumulative dividends, the self-financing condition is defined by considering the entire gain process $G$. Namely, $\theta$ must satisfy (see \cite[chapter 6, Sec. K]{duffie})
\begin{equation}
\label{eq:sfc}
d\Pi=\theta\cdot dG=\theta\cdot (dY+dD).
\end{equation}
For the dynamics given in Subsection \ref{subsec:asset-dividend} we have
\begin{equation}
\label{eq:gain diff}
dG=\begin{pmatrix}
d\V -\epsilon^H dJ^{B}\\
r^X X \, dt\\
\gamma^\K \K\, dt\\
dS+(\gamma_S-q_S)S\, dt \\
d\PC-q_C \PC \, dt\\
d \PB\\
0
\end{pmatrix}.
\end{equation}
We now substitute \eqref{eq:gain diff} and \eqref{eq:port-notation} into \eqref{eq:sfc}. As $d\Pi=0$, this leads to
\begin{equation}
\label{eq:fundamental differential}
\begin{split}
0=d\V-\epsilon^HdJ^{B}-r^X X\, dt -\gamma^\K \K\, dt +\delta \left( dS+(q_s-\gamma_S)S\, dt\right) + \aC (d\PC-q_C \PC \, dt)+\aB d \PB 
\end{split}.
\end{equation}
Next, we devote some comments about the term in $\K$ of the latter expression and analyze its financial meaning, which embodies the essence and cause of KVA. The presence of the $\gamma^\K \K$ instead of the cost $\gamma^\K \varphi \K$  means that the whole cost of $\K$ is charged to the client. Here and in \cite{green2014kva}, this does not depend on how much capital is available to the derivative desk. However, shareholders ask for a return on their entire investment, while at the same time regulations require blocking some capital inside the institution.  This quantity cannot be used to generate a profit, and thus it becomes a cost.
The next step is to eliminate risks inside \eqref{eq:fundamental differential}, i.e., consider a riskless portfolio.
\begin{itemize}
\item[1.] We first apply Itô's formula for jump-diffusion processes to $\V$ (see \cite[Ch. 14]{pascucci2011pde}, for example), and get
\begin{equation}
\label{eq:Itô V}
\begin{split}
d\V&=\frac{\partial \V}{\partial t}\, dt+\frac{\partial \V}{\partial S}\, dS +\frac{1}{2}\frac{\partial^2 \V}{\partial S^2}\, d\langle S \rangle +\Delta_C\V dJ^C+\Delta_B \V dJ^B\\
&=\frac{\partial \V}{\partial t}\, dt+\Delta_C\V dJ^C+\Delta_B \V dJ^B\\
&\quad +\mu S \frac{\partial \V}{\partial S}\, dt  +\frac{1}{2}\sigma^2 S^2 \frac{\partial^2 \V}{\partial S^2} \, dt +\sigma S \frac{\partial \V}{\partial S} dW.
\end{split}
\end{equation}
For $\beta\in \R$ we define the differential operator $\A^\beta$ given by
\begin{equation}
    \label{eq:BS op}
    \A^\beta :=\beta S \frac{\partial }{\partial S}  +\frac{1}{2}\sigma^2 S^2 \frac{\partial^2 }{\partial S^2}. 
\end{equation}
\item[2.]  We deduce the hedge ratios, i.e., the amount of $\delta$ and $\aC$ needed to cancel terms in $dW$ and $dJ^C$. This leads to
\begin{equation}
\label{eq:hedg ratio}
\aC=\frac{\Delta_C V}{(1-R^C)\PC},\qquad\delta =- \frac{\partial V}{\partial S}.
\end{equation}
\item[3.] As far as the risk $J^B$ is concerned, the terms containing the jump $dJ^{B}$ are
$$-\epsilon^H+ \Delta_B \V -\aB\PB(1-R^B).$$
By using \eqref{eq:financing constr K} and \eqref{eq:def hedg err}, the previous expression is reduced to
$$-\epsilon^H+ \Delta_B \V +(1-R^B)(\V-X-\varphi  \K)=-\epsilon^H+\epsilon^H=0.$$
As anticipated, the error $\epsilon^H$ in the funding account made of $\PB$ is paid back when $B$ defaults. Theoretically, a zero hedging error can be achieved by using the riskless account $B$. However, very often it is not possible to finance risk-free. In this work, we adopt a single-bond financing strategy, which is the most used for real-world applications. For a discussion on different financing policies, we address to \cite{burgard2013funding}. 
\item[4.] Since all risks are now hedged we deduce a PDE model by posing the drifts of \eqref{eq:fundamental differential} equal to zero. Indeed, by \eqref{eq:hedg ratio} and \eqref{eq:financing constr K} we obtain
\begin{align*}
  0=  \frac{\partial \V}{\partial t}+\frac{1}{2}\sigma^2 S^2 \frac{\partial^2 \V}{\partial S^2}+(q_S- \gamma_S) S\frac{\partial \V}{\partial S}
    +\frac{r^C-q^C}{1-R^C}\Delta_C\V
    -r^B(\V-X-\varphi\K)-r^X X-\gamma^\K \K,
\end{align*}
that is,
\begin{align*}
    \frac{\partial \V}{\partial t}+\A^{(q_S-\gamma_S)}\V
    =r^B\V+ \frac{q^C-r^C}{1-R^C}\Delta_C\V &+(r^X-r^B)X+(\gamma^{\secondReviewer{\K}}-\varphi r^B)\K. 
\end{align*}
\end{itemize}
By using \eqref{eq:spread-intensity} we finally obtain a general model for $\V$,
\begin{equation}
\label{eq:general PDE price K}
\begin{cases}
\frac{\partial \V}{\partial t}+ \A^{(q_S-\gamma_S)} \V=r^B\V -\lambda_C\Delta_C\V+(r^X-r^B)\,X
+ (\gamma^{\K}-\varphi\, r^B) \K,\\
V(T,S)=g(S),
\end{cases}
\end{equation}
where $g$ denotes the payoff function of the European derivative.
The total XVA $U$ is defined through $\V=U+V$, where $V$ denotes the risk-free derivative, and thus solves the Black-Scholes type PDE
\begin{equation}
\label{eq:BS eq}
\begin{cases}
\frac{\partial V}{\partial t}+ \A^{(q_S-\gamma_S)} V=rV,\\
V(T,S)=g(S).
\end{cases}
\end{equation}
Accordingly, we have that $U$ solves
\begin{equation}
\label{eq:general PDE XVA K}
\begin{cases}
\frac{\partial U}{\partial t}+ \A^{(q_S-\gamma_S)} U-r^B U=(r^B-r)V -\lambda_C\Delta_C(V+U)+(r^X-r^B)\,X+ (\gamma^{\K}-\varphi\, r^B) \K,\\
U(T,S)=0.
\end{cases}
\end{equation}
\subsection{PDEs with close-out conditions}
\label{subsec:close-out}
In PDEs \eqref{eq:general PDE price K}
the close-out spread $\Delta_C V$ indicates how the value of the contract changes when the counterparty defaults. A standard market hypothesis is to set \begin{equation}
    \Delta_C \V= g^C(M)-\V,
\end{equation}
for a function $g^C$, whose value depends on the mark-to-market (MTM) value of the derivative. In this model we set
\begin{equation}
\label{eq:close-out}
g^C(M):=X+R^C(M-X)^++(M-X)^-,
\end{equation}
where $M$ denotes the MTM value. Specifically, if $C$ defaults and $M-X$ is positive, then $B$ can pledge the collateral account $X$, and $C$ pays the remaining $(M-X)$ times the recovery ratio $R^C$. However, if $C$ defaults and $(M-X)$ is negative, then $B$ has to pay the whole MTM price of the contract.  To the best of our knowledge, in applications $M$ is assumed to be equal to the risk-free value $V$ in CVA and DVA calculations, while in some cases  $M=\V$ can also be considered (e.g. \cite{brigobookxva} about the recursive nature of FVA).  For instance, $M=\V$ is imposed for perfectly collateralized derivatives (see for example  \cite[Equation $(9)$]{pallavicini2012funding}). Therefore, we have that
$$M=V \Longrightarrow -\Delta_C \V=U+(1-R^C)(V-X)^+,$$
while 
$$ M=\V \Longrightarrow -\Delta_C\V=(1-R^C)(\V-X)^+.$$
As a result, for $M=\V$,  the PDE \eqref{eq:general PDE price K} is nonlinear and reads as
\begin{equation}
\label{eq:PDE price M=V cap K}
\begin{cases}
\frac{\partial \V}{\partial t}+ \A^{(q_S-\gamma_S)} \V=r^B\V +\underbrace{(q^C-r^C)}_{=\lambda_C(1-R^C)}(\V-X)^+ +(r^X-r^B)\,X
+ (\gamma^{\K}-\varphi\, r^B) \K,\\
V(T,S)=g(S),
\end{cases}
\end{equation}
and the PDE \eqref{eq:general PDE XVA K} for the XVA is then
\begin{equation}
\label{eq:PDE XVA M=V cap K}
\begin{cases}
\frac{\partial U}{\partial t}+ \A^{(q_S-\gamma_S)} U-r^B U=&(r^B-r)(V-X) + (q^C-r^C)(U+V-X)^+ \\
&+(r^X-r)\,X + (\gamma^{\K}-\varphi\, r^B) \K,\\
U(T,S)=0.
\end{cases}
\end{equation}
Analogously, for $M=V$ we obtain the following linear PDE for the derivative price:
\begin{equation}
\label{eq:PDE price M=V K}
\begin{cases}
\frac{\partial \V}{\partial t}+ \A^{(q_S-\gamma_S)} \V=(r^B +\lambda^C)\V-\lambda^C V +(q^C-r^C)(V-X)^+ +(r^X-r^B)\,X
+ (\gamma^{\K}-\varphi\, r^B) \K,\\
V(T,S)=g(S),
\end{cases}
\end{equation}
while the XVA satisfies
\begin{equation}
\label{eq:PDE XVA M=V K}
\begin{cases}
\frac{\partial U}{\partial t}+ \A^{(q_S-\gamma_S)} U-(r^B+\lambda^C) U=&(r^B-r)(V-X) + (q^C-r^C)(V-X)^+ \\
&+(r^X-r)\,X + (\gamma^{\K}-\varphi\, r^B) \K,\\
U(T,S)=0.
\end{cases}
\end{equation}
In particular, a solution of \eqref{eq:PDE XVA M=V K} can be expressed in terms of expectation by means of the Feynman-Kac representation formula (see \cite[Ch.9]{pascucci2011pde}, for example). Namely,
$$U=-\text{CVA} + \text{FBVA}-\text{FCVA}- \text{CRA}-\text{KVA},$$ 
which by \eqref{eq:spread-intensity} can be defined as
\begin{align*}
&\text{CVA}_t(S)=\E\left[ \int_t^T  \lambda^C_u(1-R^C_u) \, e^{-\int_t^u (r^B_s+\lambda^C_s)\, ds}  (V_u(S_u)-X_u(S_u))^+ \, du\Big{|} S_t=S\right],\\
&\text{FBVA}_t(S)=-\E\left[ \int_t^T  \lambda^B_u(1-R^B_u)\, e^{-\int_t^u (r^B_s+\lambda^C_s)\, ds}  (V_u(S_u)-X_u(S_u))^- \, du \Big{|} S_t=S\right],\\
&\text{FCVA}_t(S)=\E\left[ \int_t^T  \lambda^B_u(1-R^B_u)\, e^{-\int_t^u (r^B_s+\lambda^C_s)\, ds}  (V_u(S_u)-X_u(S_u))^+ \, du\Big{|} S_t=S\right],\\
&\text{CRA}_t(S)=\E\left[ \int_t^T  (r^X_u-r_u) \, e^{-\int_t^u (r^B_s+\lambda^C_s)\, ds} X_u(S_u)\, du \Big{|} S_t=S\right],\\
&\text{KVA}_t(S)=\E\left[ \int_t^T  (\gamma^\K_u-\varphi_u r^B_u) \, e^{-\int_t^u (r^B_s+\lambda^C_s)\, ds} \K_u(S_u, V_u)\, du \Big{|} S_t=S\right].
\end{align*}
With respect to the acronyms just defined, the CVA is Credit-Value-Adjustment and quantifies the loss under the possibility the counterparts might default. FBVA and FCVA are adjustments for Funding Benefits and Costs, respectively. CRA is the Collateral Rate Adjustment and reflects the expected excess of net interests paid on collateral. Finally, the KVA stands for Capital Value Adjustments and represents the expected cost of remunerating shareholders because of the cost of capital.
Analogous formulas for the case $M=\V$ can be obtained with the solution $\V$ and $U$ appearing inside the expectations.

We finally highlight that the collateral $X$ and $\K$  have not been defined yet. Concerning $X$, we set a partial collateralization
$$X=\gamma_X M,\qquad \gamma_X\in [0,1].$$ 
As far as $\K$ is concerned, there is no unique definition. Its definition depends on the type of product considered and the particular regulatory method that the institution is allowed to use. Under some restrictions, in \ref{sec:def of capital} we will show that the capital requirement generated by a European vanilla option is given by a function $\K=\K(t,S,M)$ which is globally Lipschitz in all variables.
\subsection{A different approach to the cost of capital}
\label{subsec:kva garcia}
As stated in the introduction, the cost of capital is controversial and so there is no general consensus in industry about its management. In this part, we briefly present a one-dimensional version of the PDE model for KVA of \cite{garcia2016recursive}. This model differs from the one of \cite{green2014kva} in the part concerning the capital account $CA$, and so in the cost of capital itself. Specifically, the capital account of this economy is given by
$$CA_t=\K-(\V-V^f) \quad \text{and}\quad  dD^{CA}_t=\gamma^\K (CA_t)\, dt -d(CA)_t,$$
where $V^f$ solves the PDE  
\begin{equation}
\begin{cases}
\frac{\partial V^f}{\partial t}+ \A^{(q_S-\gamma_S)} V^f=(r^B+\lambda_C) V^f -\lambda_C g^C+(r^X-r^B)\,X, \\
V^f(T,S)=g(S).
\end{cases}
\end{equation}
Namely, $V^f$ denotes the price of the derivative including counterparty credit risk and funding costs and benefits. The main idea behind this kind of capital account is that the KVA $=\V-V^f$ can be accounted as retained earnings. Since retained earnings are considered as Core Equity Tier 1 Capital (CET1), i.e., part of $\K$, the amount the bank needs to get from shareholders might be reduced. We address the reader directly to \cite{garcia2016recursive} and \cite{garcia2019retained} for more details. It is worth stressing that under the approach in \cite{garcia2016recursive} shareholders are remunerated only for the amount $CA_t=\K_t-(\V-V^f)$, whereas in the previous model of \cite{green2014kva} the whole capital (including retained earnings) is remunerated.\\
By keeping the rest of the assets as in Subsection \ref{subsec:asset-dividend}, the same passages of Subsection \ref{subsec:semi-repl} lead to the following PDE
\begin{equation}
\begin{cases}
\frac{\partial \V}{\partial t}+ \A^{(q_S-\gamma_S)} \V=\gamma^\K \V +\lambda_C(\V-g^C)+(r^X-r^B)\,X \\
\qquad \qquad \qquad \qquad  
+(\gamma^\K-r^B)(\K-V^f),\\
\V_T(S)=g(S).
\end{cases}
\end{equation}
For the choice of the MTM $M=V$, then $U=\V-V^f=\text{KVA}$ solves
\begin{equation}
\label{eq:GAR KVA}
\begin{cases}
\frac{\partial U}{\partial t}+ \A^{(q_S-\gamma_S)}U=(\gamma^\K +\lambda^C)U+ (\gamma^\K-r^B)\K,\\
U_T(S)=0.
\end{cases}
\end{equation}
The counterpart of Subsection \ref{subsec:close-out} of this PDE is
 \begin{equation}
\begin{cases}
\frac{\partial U'}{\partial t}+ \A^{(q_S-\gamma_S)} U'=(r^B+\lambda^C) U'+ (\gamma^{\K}-\varphi\, r^B) \K,\\
U'_T(S)=0.
\end{cases}
\end{equation}
In particular, when $\varphi=1$ 
we then have
$$U(t,S)= \exp\left( -\int_t^T (\gamma^K_s-r^B_s) \, ds\right) U'(t,S).$$
Namely, as long as the spread  $\gamma^K-r^B$ is positive then the KVA of Subsection \ref{subsec:close-out} is greater than the one in \eqref{eq:GAR KVA}.

\section{Mathematical analysis of the PDE model} 
\label{sec:PDE and BSDE}
In this section we address the mathematical analysis of Cauchy problems contained in Subsection \ref{subsec:close-out}. 
With regards to the alternative KVA model in Subsection \ref{subsec:kva garcia}, we have already observed that for deterministic funding rate, its solution is easily inferred from a particular case of \eqref{eq:PDE XVA M=V K}.
\\
We notice that all PDEs presented in this work are parabolic, either linear or semilinear (see \cite[Ch. 2]{renardy2006introduction}, for example). On this topic, the reader might be interested in the theory of abstract parabolic problems,  for which \cite{henry2006geometric} presents the theory of sectorial operators and in its Example 3.6 considers semilinear diffusion problems on a bounded domain $\Omega$ and initial condition in $\Leb^2(\Omega)$. Alternatively, a more analytical approach can be found in Friedman's book \cite[Ch.7]{friedman2008partial}, where classical well-posedness of the Cauchy problem is treated using the fundamental solution methodology. However, in \cite{friedman2008partial} only bounded domains are considered. More recently, the book \cite{lunardi2012analytic} contains examples of semilinear reaction-diffusion systems in the domain $[0,T]\times \R^n$ with data in $\Leb^{\infty}(\R^n)$.

In the present work, we are dealing with PDEs of the form 
\begin{equation}
\label{eq:all PDEs}
\frac{\partial \V}{\partial t}+ \A^{\beta} \V=F(t,S,\V), \qquad (t,S)\in [0,T)\times (0,\infty),
\end{equation}
where $\beta=(q_S-\gamma_S)$ and $\mathcal{A}^{\beta}$ denotes the Black-Scholes operator \eqref{eq:BS op}, which is parabolic, but the spatial domain is unbounded. Moreover, the final condition is neither bounded nor integrable for the call option. When the MTM is equal to the risk-free derivative price $V$, the expression of the term $F$ is
\begin{equation}
\label{eq:driver linear PDE}
\begin{split}
F(t,S,\V)= (r^B+\lambda^C)\V-\lambda^C V +(q^C-r^C)(V-X)^+ &+(r^X-r^B)X \\
&+ (\gamma^{\K}-\varphi\, r^B) \K(t,S,V),
\end{split}
\end{equation}
and thus $F$ is linear in $\V$ (we recall that $X=X(V)$ and $\K=\K(t,S,V)$), so that the PDE is linear.
On the other hand, when $M=\V$ we have
\begin{equation}
\label{eq:driver nonlin PDE}
\begin{split}
F(t,S,\V)=r^B\V +(q^C-r^C)(\V-X)^+ +(r^X-r^B)\,X + (\gamma^{\K}-\varphi\, r^B) \K(t,S,\V),
\end{split}
\end{equation}
 so that $F$ becomes nonlinear in $\V$ and a semilinear PDE arises. 
 
 Note that by rewriting \eqref{eq:all PDEs} in log-coordinates we obtain a semilinear reaction-diffusion-equation (see \cite[Ch.2]{galaktionov2012stability}) very similar to the ones studied in \cite[Sec. 7.3.1]{lunardi2012analytic}. Since the term $F$ is Lipschitz in $\V$ we might expect to prove mild well-posedness with fixed-point theorems. Therefore, the main difficulty of equation \eqref{eq:all PDEs} lies in the term $F$, whose nonlinearity has linear growth, in contrast with some combustion models presenting a superlinear term 
(e.g.  Frank-Kamenetskii equation).

Alternatively, the aforementioned Cauchy problem can be solved in a viscous sense by considering the equivalent system of Forward-Backward SDEs
\begin{equation}
\label{eq:FBSDE}
\begin{cases}
dS_t=(q_S-\gamma_S)S_t\, dt + \sigma S dW_t, \\
-d\V_t=F(t,S_t,\V_t) \, dt- Z_t\, dW_t, \qquad t\in [0,T],\\
S_0=S ,\, \V_T=g(S_T).
\end{cases}
\end{equation}
In the context of BSDEs, we note that the driver $F$ is Lipschtitz in $\V$ uniformly in $(t,S)$. Moreover, the process $S$ is a geometric Brownian motion and so classical conditions for well-posedness of \eqref{eq:FBSDE} are satisfied (see \cite{pardoux1990adapted}, for example). Finally, given that the final condition $g$ is continuous and has polynomial (linear) growth,  a well-known result of \cite{pardoux2005backward} shows that the solution $\V$ of \eqref{eq:FBSDE} is a viscosity solution of \eqref{eq:all PDEs}.
\subsection{Well-posedness of the PDEs formulation}
\label{subsec:well-posedness PDE}
In this section, we aim to prove the well-posedness of \eqref{eq:all PDEs} and the regularity of its solution when nonlinearities in \eqref{eq:all PDEs} are Lipschitz and the terminal condition has at most linear growth. The proposed techniques are mainly addressed to the semilinear case included in \eqref{eq:all PDEs}. Although the linear case can be understood as a particular case of the semilinear one, the mathematical analysis can be straightforwardly performed in the frame of linear PDEs. Concerning the semilinear case, the well-posedness techniques can be framed into the semigroup theory provided by \cite{henry2006geometric} and \cite{lunardi2012analytic}, for example. Next, for the regularity of the solution, we deploy the Gaussian estimates associated with fundamental solutions as in \cite[Ch. 20]{pascucci2024pde}.

Note that similar results can be found in \cite{arregui2019mathematical} for the one-factor case and in \cite{arregui2018mathematical} for two factors in XVA models without considering KVA, in which well-posedness is deduced from the sectorial property of the Laplace operator (see \cite{henry2006geometric}). Our approach is slightly different, and it takes advantage of the explicit expression of the semigroup associated with the PDE. Thus, by direct computation on the semigroup, we also obtain some additional regularity results for the solution. 

We start by noting that the solution of the classical Black-Scholes PDE 
\begin{equation}
\label{eq:BS pde zero interest rate}
\begin{cases}
\frac{\partial \V}{\partial t}+ \A^{\beta} \V=rV, \qquad \text{in } Q_T:=(0,T)\times \R_{\geq 0}\\
\V(T,S)=(S-K)^+,
\end{cases}
\end{equation}
satisfies the relation
\begin{equation}
\label{eq:def norm}
\|\V\|_{\mathcal{X}}:=\esssup_{(t,S)\in Q_T}\Big{|}\frac{\V(t,S)}{1+S}\Big{|}<\infty,
\end{equation}
so that the risk-free value of a European call option grows linearly in $S$.
Accordingly, in order to set up a Picard iteration with the semigroup of \eqref{eq:all PDEs} we consider the Banach space
\begin{equation}
    \mathcal{X}:=\left\{ x:[0,T]\times \R^{+}\longrightarrow \R: \frac{x}{1+S}\in \Leb^\infty([0,T]\times \R^{+}) 
    \right\}
\end{equation}
endowed with $\|\cdot\|_\mathcal{X}$ norm defined in \eqref{eq:def norm}.

Next, by introducing $\Gamma(t,S,u,z)$ to denote the fundamental solution of the differential operator $\partial_t+ \A^{\beta}$,  we can consider the mild equation
\begin{equation}
    \label{eq:mild equation}
    \V(t,S)=\int_{0}^\infty \Gamma(t,S,T,z)g(z)dz- \int_t^T \int_0^\infty \Gamma(t,S,u,z) F(u,z, \V(u,z)) \, dz \, du,
\end{equation}
or equivalently, 
$$ \V(t,S)=\E[g(S_T)|S_t=S]- \int_t^T \E\left[ F(u,S_u, \V(u,S_u))|\, S_t=S\right] \, du.$$
Given the pure probabilistic form of this last expression, this kind of problem has been recently tackled through a multi-level Picard iteration technique, and convergence was also established in the multidimensional case  (see \cite{simonella2023xva} and references therein). 

Here we use this same approach to show the existence and uniqueness of the solution of \eqref{eq:mild equation} in $(\mathcal{X},\|\cdot\|_\mathcal{X})$. For simplicity, we assume $\beta$ and $\sigma$ to be constant and work in one dimension. 
Thus, we have
\begin{equation}
    \Gamma(t,S,u,z)=\frac{z^{-1}}{\sqrt{2\pi \sigma^2 (u-t)}}\exp\left( -\frac{\left[\ln{\frac{z}{S}}-(\beta-\frac{\sigma^2}{2})(u-t)\right]^2}{2\sigma^2 (u-t)} \right)\mathds{1}_{z> 0}.
\end{equation}
In the following proof, we use $c$ to denote a general positive constant, whose dependencies will be specified when it is required. We also adopt the notation 
$$m_{t,u}:= \left(\beta-\frac{\sigma^2}{2}\right)(u-t),\qquad \sigma^2_{t,u}=\sigma^2(u-t),\qquad 0\leq t< u\leq T.$$
\begin{teo}
\label{teo:existence mild}
Let $F:[0,T]\times \R_{\geq 0}\times \R \longrightarrow \R$, and $g:\R_{\geq 0}\rightarrow \R$ be Lebesgue measurable and with linear growth.
Suppose that there exists $L>0$ such that
\begin{equation}
\label{eq:lipsch condition}
    |F(t,S,f)-F(t,S',g)|\leq L(|S-S'|+ |f-g|) \qquad f,g\in \R,\, (t,S)\in [0,T]\times \R_{\geq 0}.
\end{equation} 
Then $J:\mathcal{X}\rightarrow \mathcal{X}$
\begin{equation}
    \label{eq:def J}
    J(x):= \int_{0}^\infty \Gamma(t,S,T,z)g(z)dz- \int_t^T \int_0^\infty \Gamma(t,S,u,z) F(u,z, x(u,z)) \, dz \, du
\end{equation}
is well-defined, and there exists $k\in \N$ such that the composition of $J$ k-times, $J^k$, satisfies
\begin{equation}
\label{eq:local contr}
 \|J^k(x)-J^k(y)\|_\mathcal{X}< \|x-y\|_\mathcal{X}.
 \end{equation}
Accordingly, \eqref{eq:mild equation} admits a unique solution in $\mathcal{X}$.
\end{teo}
\begin{proof}
Since $F$ is uniformly Lipschitz in the $(S,x)$ variables it holds
\begin{equation}
|F(t,S,x)|\leq c (1+|S|+|x|),\qquad (t,S,x)\in [0,T]\times \R_{\geq 0}\times \R,
\end{equation}
and so, for $x\in \mathcal{X}$ there exists $c=c(F,T,\|x\|_\mathcal{X})$ such that
\begin{equation}
\label{eq:lin growth in X}
|F(t,S,x(t,S))|\leq c(1+|S|),\qquad (t,S)\in [0,T]\times \R_{\geq 0}.
\end{equation}
Accordingly,
\begin{equation*}
\begin{aligned}    
  |J(x)(t,S)|&= \Bigg{|}\int_0^\infty \Gamma(t,S,T,z)g(z)\, dz+\int_t^T \int_0^\infty \Gamma(t,S,u,z)F(u,z,x(u,z))\, dz\, du\Bigg{|}\\
  &\leq c\left( \int_0^\infty (1+z)\Gamma(u,S,T,z)\, dz+ \int_t^T \int_0^\infty (1+z)\Gamma(t,S,u,z)\, dz\, du\right).
\end{aligned}
\end{equation*} 
Since $\Gamma(t,S,u,z)$ is the density function of a geometric Brownian motion at time $u\geq t$ and initial condition $S$ we get
\begin{equation}
\label{eq:proof bound norm x}
\begin{aligned}
|J(x)(t,S)|&\leq c\left( (1+ S\exp\left(\beta(T-t)\right) + \int_t^T (1+ S\exp\left(\beta(u-t)\right)) du\, \right)\\
&= c \left( 1+ S \exp{(\beta(T-t))}+(T-t)\left[ 1+ S \frac{\exp\left(\beta(T-t)\right)-1}{\beta (T-t)}\right]\right)\\
&\leq c(1+(T-t))(1+S)\leq c(1+S),
\end{aligned}
\end{equation}
and this proves the well-posedness of $J(x)$.
To prove \eqref{eq:local contr} we see that
\begin{equation}
\begin{aligned}
|J(x)-J(y)|(t,S) & \leq \int_t^T \int_0^\infty \Gamma(t,S,u,z)\big{|} F(u,z,x(u,z))-F(u,z,y(u,z))\big{|}\, dz\, du \\
&\leq c \,\int_t^T \int_0^\infty \Gamma(t,S,u,z)|x-y|(u,z)dz\, du\\
&\leq c\|x-y\|_\mathcal{X}\int_t^T \int_0^\infty \Gamma(t,S,u,z)(1+z)dz\, du.
\end{aligned}
\end{equation}
By repeating the computations in \eqref{eq:proof bound norm x} we thus obtain
\begin{equation}
\label{eq:proof quasi local contr}
\esssup_{S\geq 0} \frac{|J(x)-J(y)|(t,S)}{1+S}\leq c(T-t) \|x-y\|_{\mathcal{X}},
\end{equation}
for some $c>0$ independent of $x,y$.
The statement now follows by standard induction arguments applied to \eqref{eq:proof quasi local contr}.
\end{proof}
\begin{cor}
\label{cor:existence classical sol}
Under the assumptions of Theorem \ref{teo:existence mild}, the mild solution $\V$ of \eqref{eq:all PDEs} satisfies
\begin{align}
\label{eq:bound first der}
& \frac{S}{1+S}\Big{|}\frac{\partial \V}{\partial S}\Big{|}(t,S) \leq c, \qquad (t,S)\in [0,T)\times \R^{+},\\
\label{eq:limit second der}
&\lim_{S\to \infty} \frac{\partial^2 \V}{\partial^2 S}(t,S)=0, \qquad t\in [0,T).
\end{align}
\end{cor}
\begin{proof}
In this proof, we make use of the following Gaussian inequality (see \cite[Lemma 20.3.4]{pascucci2024pde}). Namely, for every $\lambda_0>\lambda>0$ and $p>0$ there exists a constant $c=c(p,\lambda,\lambda_0)$ such that
\begin{equation}
\label{eq:gaussian est}
\left(\frac{|x|}{\sqrt{t}}\right)^p\frac{1}{\sqrt{2\pi \lambda t}}\exp\left(-\frac{x^2}{2\lambda t}\right)\leq  \frac{c}{\sqrt{2\pi \lambda_0 t}}\exp\left(-\frac{x^2}{2\lambda_0 t}\right), \qquad (t,x)\in (0,\infty)\times \R. 
\end{equation}
From Theorem \ref{teo:existence mild} the mild equation \eqref{eq:mild equation} admits a unique solution with linear growth $\V$.
In order to prove regularity, we differentiate \eqref{eq:mild equation} and manage with the semigroup given by $\Gamma$. Without loss of generality, we prove regularity only for the second term of the mild equation \eqref{eq:mild equation}, and thus we assume $g=0$.
In order to prove the existence and regularity of the first derivative we show that
\begin{align*}
\int_{t}^T \Bigg{|}\int_0^\infty  \frac{\partial \Gamma}{\partial S}(t,S,u,z) F(u,z,\V(u,z))\, dz\Bigg{|} du <\infty, \qquad (t,S)\in [0,T)\times \R_{> 0}.
\end{align*}
For this purpose, by \eqref{eq:lipsch condition} and $\V\in \mathcal{X}$ it is sufficient to show that for every $S>0$
\begin{equation}
    \label{eq:proof bound firs der}
\int_0^\infty  \Big{|}\frac{\partial \Gamma}{\partial S}(t,S,u,z)\Big{|} (1+z)\, dz\, du\leq \frac{c(S)}{\sqrt{u-t}}, \qquad u\in [t,T] .
\end{equation}
For $z>0$ and $S>0$ we have
\begin{equation}
    \label{eq:proof first der Gamma}
    \frac{\partial \Gamma}{\partial S}(t,S,u,z)=\frac{\Gamma(t,S,u,z)}{ S \sigma_{t,u}} \left(\frac{\ln{\frac{z}{S}}-m_{t,u}}{\sigma_{t,u}}\right).
\end{equation}
By replacing \eqref{eq:proof first der Gamma} in \eqref{eq:proof bound firs der} and using the change of variable $z=S \exp(\sigma_{t,u}y +m_{t,u})$ in \eqref{eq:proof bound firs der}, we obtain
\begin{equation}
    \begin{aligned}
        S\int_0^\infty  \Big{|}\frac{\partial \Gamma}{\partial S}(t,S,u,z)\Big{|} (1+z)\, dz & =\frac{1}{\sigma_{t,u}}\int_{-\infty}^\infty \frac{\exp(-\frac{y^2}{2})}{\sqrt{2\pi}}\, |y|(1+S\exp(\sigma_{t,u}y +m_{t,u}))\, dy\\
        &\leq \frac{c}{\sigma_{t,u}}\left( 1 + S\int_{-\infty}^\infty \frac{|y|\,\exp(-\frac{y^2}{2})}{\sqrt{2\pi}}\, \exp(\sigma_{t,u}y +m_{t,u}))\, dy\right)\\
        &\leq \frac{c}{\sigma_{t,u}}\left( 1+S\right)\leq \frac{c(1+S)}{\sqrt{u-t}}.
    \end{aligned}
\end{equation}
This last relation proves \eqref{eq:bound first der}, which in turn implies that $\V$ is locally Lipschitz in $S$ uniformly in $t$. 

Next, in order to prove the existence of the second derivative we will consider
$$H(u,z):=F(u,z,\V(u,z)).$$
As $F$ is uniformly Lipscthitz in the two last variables, by \eqref{eq:bound first der} we obtain
\begin{equation}
\label{eq:proof loc lipsc H}
|H(u,z)-H(u,y)|\leq c_n |z-y|, \qquad u\in [0,T],\, z,y \geq n^{-1},
\end{equation}
where $c_n$ grows linearly with respect to $n$.
In order to deduce the existence of the second derivative for every $S>0$, it is sufficient to prove that 
\begin{equation}
\label{eq:proof bound sec der}
|I|:=\Bigg{|}\int_0^\infty  \frac{\partial^2 \Gamma}{\partial S^2}(t,S,u,z)H(u,z)\, dz \, dz\Bigg{|}\leq \frac{c(S)}{\sqrt{u-t}}.
\end{equation}
From a direct computation of $ \displaystyle \frac{\partial^2 \Gamma}{\partial S^2}(t,S,u,z)$ and the change of variable $z=S \exp(y+m_{t,u})$, we obtain
\begin{equation}
    \begin{aligned}
        I&=\frac{1}{S^2\sigma^2_{t,u}}\int_{-\infty}^\infty \frac{\exp\left(-\frac{y^2}{2\sigma^2_{t,u}}\right)}{\sqrt{2\pi \sigma^2_{t,u}}}\left(\frac{y^2}{\sigma^2_{t,u}}-1-\frac{y}{\sigma_{t,u}}\right)H(u,S\exp(y+m_{t,u}))\, dy\\
        &=\frac{1}{S^2\sigma^2_{t,u}}\int_{-\infty}^\infty G(y,\sigma_{t,u})H(u,S\exp(y+m_{t,u}))\,dy:=I_1+I_2,
    \end{aligned}
\end{equation}
where
\begin{equation}
\begin{split}
I_1&=\frac{1}{S^2\sigma^2_{t,u}}\int_{|y|<n} G(y,\sigma_{t,u})H(u,S\exp(y+m_{t,u}))) \, dy, \\
I_2&=\frac{1}{S^2\sigma^2_{t,u}}\int_{|y|\geq n} G(y,\sigma_{t,u})H(u,S\exp(y+m_{t,u}))\, dy.
\end{split}
\end{equation}
In order to prove \eqref{eq:proof bound sec der} for the module of $I_2$, we use that $H(u,z)$ grows linearly in $z$ and $|y|>n$. In this case,
\begin{equation}
    \begin{aligned}
      |I_2|  \leq \frac{c}{S^2\sigma^2_{t,u}}\int_{|y|\geq n} |y|G(y,\sigma_{t,u})(1+S\exp(y+m_{t,u}))\, dy,
    \end{aligned}
\end{equation}
and thus the relation follows by the estimate \eqref{eq:gaussian est}.
In order to estimate $I_1$ we observe that for $|y|\leq n$, by \eqref{eq:proof loc lipsc H} and $S>0$ we can find a constant $c=c(S,n)$ such that
\begin{equation}
\label{eq:proof est diff H}
|H(u,S\exp(y+m_{t,u}))-H(u,S)|\leq c S\,|\exp(y+m_{t,u})-1|\leq c\, S |y+m_{t,u}|. 
\end{equation}
Accordingly,
\begin{equation}
    \begin{aligned}
        |I_1|\leq |J_1|+|J_2|=&\frac{1}{S^2\sigma^2_{t,u}}\Bigg{|}\int_{|y|<n} G(y,\sigma_{t,u})\left[H(u,S\exp(y+m_{t,u})))-H(u,S)\right] \, dy \Bigg{|} \\
        &+ \frac{|H(u,S)|}{S^2\sigma^2_{t,u}}\Bigg{|}\int_{|y|<n} G(y,\sigma_{t,u}) \, dy \Bigg{|}.
    \end{aligned}
\end{equation}
The estimate \eqref{eq:proof bound sec der} for $|J_1|$ follows by \eqref{eq:proof est diff H} and the Gaussian estimate \eqref{eq:gaussian est}. In particular, we obtain
$$|J_1|\leq \frac{c(S)}{S\sqrt{u-t}}$$
with $\lim_{S\to 0^+} c(S)=\infty$ and $c(S)$ bounded in $S>1$.
Finally, to bound $|J_2|$, notice that
$$\int_{-\infty}^{\infty} G(y,\sigma_{t,u})\, dy=\int_{-\infty}^{\infty} \frac{\exp(-\frac{z^2}{2})}{\sqrt{2\pi} }(z^2-1-z)\, dz=0,$$
and thus
\begin{align*}
    |J_2|&=\frac{|H(u,S)|}{S^2\sigma^2_{t,u}}\Bigg{|}\int_{|y|>n} G(y,\sigma_{t,u}) \, dy \Bigg{|}=\frac{|H(u,S)|}{S^2\sigma^2_{t,u}}\Bigg{|}\int_{|z|>\frac{n}{\sigma_{t,u}}}\frac{\exp(-\frac{z^2}{2})}{\sqrt{2\pi} }(z^2-1-z)\, dz\Bigg{|}.
\end{align*}
In order to estimate the integral, without loss of generality, we can assume $\sigma_{t,u}<\frac{1}{2}$ and so
$$\frac{1}{2}\left(z+\frac{n}{\sigma_{t,u}}\right)<z< \frac{z^2}{2},\qquad \left \{|z|>\frac{n}{\sigma_{t,u}}\right \}\subseteq \{|z|>n\}.$$
As a consequence,
\begin{align*}
   |J_2|&\leq \frac{c}{S^2\sigma^2_{t,u}} \exp\left({-\frac{n}{2\sigma_{t,u}}}\right) \int_{|z|>n} \exp\left(-\frac{z}{2}\right)|z^2-1-z|\, dz \leq \frac{c}{S^2} \frac{\exp(-\frac{n}{2\sigma_{t,u}})}{\sigma^2_{t,u}}\leq \frac{c}{S^2 \, \sigma_{t,u}},
\end{align*}
as the exponential decay of $\exp(-\frac{n}{2\sigma_{t,u}})$ dominates over  $\sigma^{-2}_{t,u}$ when $u\to t^+$. Indeed such a limit for $|J_2|$ is zero. 
We highlight that as for the other terms, the constant $c=c(S,n)$ diverges only when $S$ tends to zero. This proves the existence of the second derivative of $\V$ along with \eqref{eq:limit second der}, and so the proof is now concluded.
\end{proof}
Note that \eqref{eq:limit second der} will be used to impose the boundary conditions in the computational domain to be defined for the numerical solution.

\section{Numerical methods for \secondReviewer{the arising} PDEs} \label{sec:numMethods}
In this section we describe the proposed numerical methods to solve the (non)linear parabolic PDE problems, such as PDE \eqref{eq:all PDEs}. As in many problems in finance, in a previous step to the numerical solution, the initial spatial unbounded domain is truncated to a computational bounded one. For this purpose, we introduce a large enough fixed value of the asset price $\overline{S}$ and consider the spatial bounded domain $[0,\overline{S}]$. Moreover, we reformulate the previously posed PDE problems \eqref{eq:all PDEs} with final condition in terms of equivalent initial PDE problems. Thus, we consider the new time to maturity variable $\tau=T-t$.

Next, following the approach contained \secondReviewer{in} \cite{wang2016stability}, we combine an IMplicit-EXplicit (IMEX) method for the time discretization with a Local Discontinuous Galerkin (LDG) method for the spatial discretization. Thus, we describe the LDG scheme in \secondReviewer{S}ubsection \ref{subsubsec:LDGdiscr}, and a couple of IMEX time-marching schemes in \secondReviewer{S}ubsection \ref{subsubsec:IMEXdiscr}. LDG methods are suitable for nonlinear problems that can be written in conservative form. We address the reader to \cite{cockburn2012discontinuous} for the general theory about LDG, and \cite[Part III]{bertoluzza2009numerical} for the treatment of schemes with alternating fluxes. In finance, documented LDG methods combined with explicit Runge-Kutta schemes have been used for portfolio optimization problems and can be found in \cite{liu2007modeling} and \cite{birge2015local}. With respect to IMEX schemes used in finance, we point out the research developed in the direction of option pricing with jump-diffusion processes. For this kind of problems, the nonlocal integral term is treated in explicit form, while the diffusion and advection terms are integrated implicitly. A presentation of different schemes and stability analysis can be found in \cite{salmi2014imex1}. In two spatial dimensions, we can mention the case with stochastic volatility under the Bates model that is considered in \cite{salmi2014imex2}. Further works in dimension two are centered on operator splitting schemes in order to combine the ADI method with IMEX time integration (e.g. \cite{boen2020operator,boen2021operator,in2018adi}). 
\subsection{LDG space semidiscretization} 
\label{subsubsec:LDGdiscr}
After truncating the domain and introducing the new time variable $\tau=T-t$ the problem \eqref{eq:all PDEs} can be written in conservative form as
\begin{equation}
\label{eq:generalConserPDE}
\begin{cases}
\partial_\tau \V + \partial_S  f(S,\V) =\partial_S\left(a(S)\partial_S \V \right)+ H(\tau,S,\V), \qquad (\tau,S)\in (0,T]\times [0,\bar{S}],\\
\V(0,S)=g(S),
\end{cases}
\end{equation}
where
\begin{equation}
\begin{cases}
a(S)=\frac{1}{2}\sigma^2 S^2, \\
f(S,\V)=(\sigma^2-\beta)S \, \V,\\
H(\tau,S,\V)=(\sigma^2-\beta)\V-F(T-\tau,S,\V),
\end{cases}
\end{equation}
and $F$ is defined as in \secondReviewer{\eqref{eq:driver linear PDE}-\eqref{eq:driver nonlin PDE}. This is the formulation we will solve numerically to obtain the results that are shown in the next section.}
%
We now display the spatial discretization of \eqref{eq:generalConserPDE}
%
along with conditions on the parabolic boundary 
$$
\{0\}\times [0,\Smax] \cup (0,T]\times \{0\} \cup (0,T]\times \{\Smax\}.$$
The initial condition is given by the payoff function $g$, that is
\begin{equation}
 \V(0,S) = g(S),\quad S\in[0,\Smax]. \label{eq:pdeIniCond}
\end{equation}
On the sides of the boundary, motivated by Corollary \ref{cor:existence classical sol}, for a call option, we impose the condition
\begin{equation} 
\label{eq:callBC}
 \V(\tau,0) = 0, \quad \partial_{SS} \V(\tau,\Smax) = 0, \quad \tau\in(0,T),
\end{equation}
while for a put option we impose
\begin{equation} \label{eq:putBC}
 \partial_{SS}\V(\tau,0) = 0, \quad \V(\tau,\Smax) = 0, \, \quad \tau\in(0,T).
\end{equation}
In order to pose the appropriate formulation to apply a discontinuous Galerkin method, we introduce the new unknown 
\begin{equation} \label{new_unknown}
    Q(\tau,S) = \partial_S \V(\tau,S),
\end{equation}
and the following notation $G(S,Q)=a(S)Q$.
Thus, the second order PDE \eqref{eq:generalConserPDE} can be equivalently formulated in terms of \secondReviewer{the} following first-order system:
\begin{equation}
\label{eq:system_PDE}
\begin{cases}
 \partial_\tau\V + \partial_Sf(S,\V) = \partial_S G(S,Q) + H(\tau,S,\V), \quad (\tau,S)\in [0,\Smax]\times(0,T],\\
 Q(\tau,S) = \partial_S \V(\tau,S), 
 \end{cases}
\end{equation}
with the same initial condition \eqref{eq:pdeIniCond} and boundary conditions \eqref{eq:callBC} or \eqref{eq:putBC}.

Next, for the spatial discretization of system \eqref{eq:system_PDE} with a LDG method,
we consider the \secondReviewer{uniform} mesh
$$\mathcal{T}_h:=\{ I_j=(S_j,S_{j+1}]\secondReviewer{=(jh,(j+1)h]},\, 0\leq j<N\}$$ associated to the set of nodes
$0=S_0<S_1<\cdots < S_N=\Smax$, \secondReviewer{with constant cell width $h=\Smax/N$}. 
Associated to the previous mesh, we consider the discontinuous finite element space of piecewise polynomial function having  degree at most $k$
$$E_h := \{v \in \Leb^2([0,\Smax]): v\lvert_{I_j}=v_j\in \mathcal{P}_k(I_j), \forall j=0,\ldots,N-1\}.$$
For the internal product of $\Leb^2((0,\Smax])$ we use the notation
$$\ip{v}{w}=\sum_{0\leq j<N}\int_{I_j} v(s)w(s)\, ds=\sum_{0\leq j<N} \ip{u_j}{v_j}_j,$$
and a general element $v\in E_h$ has the form
\begin{equation*}
v=\sum_{0\leq j<N}\sum_{i=0}^k v_{j}^{i} \Phi_{j}^{i},
\end{equation*}
where for each $j$, $\{ \Phi_{j}^{i}, 0\leq i\leq k\}$ is a basis of $\mathcal{P}_k(I_j)$. 
Having the elements of a basis of $E_h$ compact support, $v\in E_h$ can present discontinuities across the edges of the cells. Therefore, there are two traces along the right-hand and left-hand of each cell, here denoted by $v^+$ and $v^-$, respectively.

For a given $\tau>0$, the semidiscrete LDG scheme aims to find the numerical solution $(\Vnum(\tau,.),q(\tau,.)) \in E_h\times E_h$
\begin{equation}
\label{eq:Vnum as linear combination}
 \Vnum(\tau,.)=\sum_{0\leq j<N}\sum_{i=0}^k \Vnum_{j}^{i}(\tau) \Phi_{j}^{i}, \qquad  q(\tau,.)=\sum_{0\leq j<N}\sum_{i=0}^k q_{j}^{i}(\tau) \Phi_{j}^{i},
\end{equation}
such that
\begin{equation}
\label{eq:varForm}
\begin{aligned}
 &\ip{\partial_\tau \Vnum}{v}_j = \D_j(q,v)+ \C_j(\Vnum,v) + \Hv_j(\Vnum,v),  \\
 & \ip{q}{w}_j = \Kv_j(\Vnum,w), 
 \end{aligned}
\end{equation}
for each cell $I_j$
and every $(v,w)\in E_h\times E_h$. Unless strictly required, to the aim of simplicity  in the following
we omit the use of the variables $\tau$ and $S$. In \eqref{eq:varForm}, we use the following notation:
\begin{equation}
\label{eq:CHDK}
\begin{aligned}
 \C_j(\Vnum,v) &= \ip{f(\Vnum)}{\partial_S v}_j - \tilde{f}_{j+1}\, v(S_{j+1}^-) + \tilde{f}_{j}\, v(S_{j}^+), \\
 \Hv_j(\Vnum,v) &= \ip{H(\Vnum)}{v}_j, \\
 \D_j(q,v) &= -\ip{G(q)}{\partial_S v}_j + G(\tilde{q}_{j+1})\, v(S_{j+1}^-) - G(\tilde{q}_{j})\, v(S_{j}^+),\\
   \Kv_j(u,w) &= -\ip{u}{\partial_S w}_j + \tilde{u}_{j+1} w(S_{j+1}^-) - \tilde{u}_{j} w(S_{j}^+),
\end{aligned}
\end{equation}
where $\tilde{f},\tilde{u}$ and $\tilde{q}$ are numerical fluxes associated to $f,u$ and $q$ respectively. Although $\tilde{f}$ can be any monotone numerical flux, in this work we choose the simple Lax-Friedrich flux
$$\tilde{f}_j = \tilde{f}(\Vnum(S_j^-),\Vnum(S_j^+))=\frac12 \bigg( f(\Vnum(S_j^-)) + f(\Vnum(S_j^+)) - \alpha_j \big( \Vnum(S_j^+) - \Vnum(S_j^-) \big) \bigg), \quad \mbox{with} \; \alpha_j = \max_{\Vnum \in I_j}\lvert \partial_{\Vnum} f(\Vnum) \rvert.$$
Secondly, for $\tilde{u}$ and $\tilde{q}$ an alternating numerical flux has to be considered. The crucial point with the alternating is that $\tilde{u}$ and $\tilde{q}$ have to be chosen from different directions:
\begin{enumerate}[label=A\arabic*]
 \item Selecting $\tilde{u}$ from the left and $\tilde{q}$ from the right, i.e. $\tilde{u}_j=u_j^-$ for $j=1,\ldots,N$ and $\tilde{q}_j=q_j^+$ for $j=0,\ldots,N-1$. This selection is well-suited for imposing the boundary conditions for call options written in \eqref{eq:callBC}. On the one hand, we force homogeneous Dirichlet boundary conditions at $S=0$, i.e. $\tilde{u}_0=0$. On the other hand, by enforcing a constant behavior of $q$ in the neighborhood of $\Smax$
 i.e $\tilde{q}_N=q_N(\Smax-\varepsilon)$,  $\varepsilon>0$, 
we impose $\partial_{SS}\V(\Smax)=0$.
\item Taking $\tilde{u}$ from the right and $\tilde{q}$ from the left, i.e. $\tilde{u}_j=u_j^+$ for $j=0,\ldots,N-1$ and $\tilde{q}_j=q_j^-$ for $j=1,\ldots,N$. This selection is suitable in order to impose the boundary conditions for put options reported in \eqref{eq:putBC}. Firstly, we force homogeneous the Dirichlet boundary condition at $S=\Smax$, i.e. $\tilde{u}_N=0$. Then, $\tilde{q}_0=q_0(\varepsilon)$, $\varepsilon>0$, allows to impose $\partial_{SS}\V(0)=0$.
\end{enumerate}

\secondReviewer{The number of equations in the variational formulations \eqref{eq:varForm} is equal to $2N(k+1)$.} By summing up these variational formulations over all the cells we get the following semidiscrete LDG in global form
\begin{equation}
\label{eq:global variational form}
\begin{aligned}
 &\ip{\partial_\tau \Vnum}{v} = \D(q,v)+ \C(\Vnum,v) + \Hv(\Vnum,v) ,  \\
 &\ip{q}{w} = \Kv(\Vnum,w), 
\end{aligned}
\end{equation}
where 
$$C(\Vnum,v) = \sum_j \C_j(\Vnum,v),$$ and similarly for $\Hv$, $\D$ and $\Kv$.
In this work, we consider the orthogonal nodal basis defined by the Lagrange interpolation polynomial basis over the $k+1$ Gauss-Legendre quadrature nodes in the interval $I_j$. Specifically, for every $i,j$ we consider the canonical $i^{th}$-Lagrange polynomial  $\phi^{i}:[-1,1]\rightarrow \R$ based on the $i$-th Gauss-Legendre quadrature node $\xi_i\in [-1,1]$. Let $w_{i}$ denote the weight associated with such quadrature node. By means of the bijection 
$$T_j:[-1,1]\rightarrow [S_j,S_{j+1}],\qquad \xi\mapsto \frac{S_{j+1}+S_j}{2} + \frac{\secondReviewer{h}}{2}\xi,$$
the basis element is defined as $\Phi_{j}^{i} :=  \phi^{i}\circ T_j^{-1}$.
As an example, we see that under this setting it holds
\begin{equation}
\label{eq:source term form}
\begin{aligned}
 \Hv_j(\Vnum,\Phi_{j}^{i}(S)) &= \int_{I_j} H(\Vnum)\Phi_{j}^{i}(S) \,dS \;
  = \; \dfrac{\secondReviewer{h}}{2}\int_{-1}^1 H \left( \Vnum \left( T_j(\xi) \right) \right) \, \phi^{i}(\xi)\,d\xi \\
 &\approx \dfrac{\secondReviewer{h}}{2} \sum_{i=0}^{k} w_{i} H (  \Vnum ( T_j(\xi_i) ) ) \, \phi^i(\xi_i) 
  \; = \; \dfrac{\secondReviewer{h}}{2} w_{j} H ( \Vnum_{j}^{i} ).
\end{aligned}
\end{equation}
The computation for the remaining terms of \eqref{eq:CHDK} are left to the reader.

The initial condition $\V_0(S)\in E_h$ is taken as an approximation of the given initial solution $g(S)$. In order to avoid errors due to the projection of the payoff on the described basis, the strike price $K$ must be a node of the mesh $\mathcal{T}_h$. This can be easily achieved by properly selecting $\Smax$, which should be also big enough to impose $\partial_{SS}\V(\Smax)=0$. Due to the expression of \eqref{eq:source term form}, the variational problem \eqref{eq:global variational form} can present non-linearities in $\Vnum$. However, solving nonlinear equations is avoided by coupling the spatial discretization with the IMEX time marching scheme. As we will show in the following section, this scheme treats explicitly the terms \eqref{eq:source term form} and $\mathcal{C}$, while it treats implicitly the remaining terms.
\subsection{IMEX time semidiscretization} \label{subsubsec:IMEXdiscr}
In this section, we present the fully-discrete LDG-IMEX method. Let
$$0=\tau^0<\cdots < \tau^n<\cdots  < \tau^L=T$$
be a uniform mesh of $[0,T]$ with constant step \secondReviewer{$\delta=T/L$, where $L$ is the number of time nodes. Moreover,} let $(\Vnum^n,q^n)$ denote $(\Vnum(\tau^n), q(\tau^n))$. Given $(\Vnum^n,q^n)$, the scheme produces a numerical solution at the next time level $\tau^{n+1}$ through intermediate numerical solutions denoted as $(\Vnum^l,q^l)$. We now provide the main steps required for the second and third-order schemes.

For the second-order scheme, for any function $(v,w)\in E_h\times E_h$, we consider the LDG method with the L-stable, two-stage DIRK(2,2,2) IMEX scheme given in \cite{ascher1997implicit}. The scheme is defined using the constants
$$\gamma = 1-\frac{\sqrt{2}}{2},\qquad 
\kappa=1-\frac{1}{2\gamma}.$$
The next time level solution $(\Vnum^{n+1},q^{n+1})$ is obtained after solving two linear systems. The first stage system, defined at the intermediate time $\tau^n+\delta \gamma$ is given by
\begin{equation}
\begin{cases} \label{eq:IMEXO2_stage1}
 \ip{\Vnum^{n,1}}{v} &= \ip{\Vnum^{n}}{v} + \delta\Big[ \gamma \D(q^{n,1},v) + \gamma (\C+\Hv)(\Vnum^n,v)\Big], \\
  \ip{q^{n,1}}{w} &= \Kv(\Vnum^{n,1},w), 
 \end{cases}
\end{equation}
and subsequently,  the next time-level solution is obtained by solving
  \begin{equation}
  \begin{cases} \label{eq:IMEXO2_stage2}
 \ip{\Vnum^{n+1}}{v} &= \ip{\Vnum^{n}}{v} + \delta\Big{[}  (1-\gamma)\D(q^{n,1},v) + \gamma \D(q^{n+1},v)  \\
 & \quad + \kappa (\C+\Hv)(\Vnum^n,v) + (1-\kappa)(\C+\Hv)(\Vnum^{n,1},v)\Big{]}, \\
 \ip{q^{n+1}}{w} &= \Kv(\Vnum^{n+1},w).
 \end{cases}
 \end{equation}

Finally, the third-order IMEX scheme is taken from \cite{calvo2001linearly}, and it is defined \secondReviewer{employing} the constants
\begin{align*}
\gamma=\frac{1767732205903}{4055673282236} ,&\qquad  \beta_1 = -\frac32\gamma^2+4\gamma-\frac14, \qquad \beta_2=\frac32\gamma^2-5\gamma+\frac54,\\
& \alpha_1=-0.35,  \qquad \alpha_2 = \dfrac{\frac13 - 2 \gamma^2 - 2\beta_2 \alpha_1 \gamma}{\gamma (1-\gamma)}.
\end{align*}
The IMEX-LDG scheme consists of solving three stages of linear systems. Namely, at time  $\tau^{n}+\delta \gamma$ we have 
\begin{equation}
\begin{cases} \label{eq:IMEXO3_stage1}
    \ip{\Vnum^{n,1}}{v} &= \ip{u^{n}}{v} + \delta\Big[ \gamma \D(q^{n,1},v)+ \gamma (\C+\Hv)(u^n,v) \Big], \\
  \ip{q^{n,1}}{w} &= \Kv(\Vnum^{n,1},w).
\end{cases}
\end{equation}
Subsequently, at time $\tau^{n}+\delta \frac{1+\gamma}{2}$ we solve
\begin{equation}
\begin{cases} \label{eq:IMEXO3_stage2}
 \ip{\Vnum^{n,2}}{v}= &\ip{\Vnum^{n}}{v}+ \delta\Big[ \frac{1-\gamma}{2} \D(q^{n,1},\Vnum) + \gamma \D(q^{n,2},\Vnum) \\
 & \left( \frac{1+\gamma}{2} - \alpha_1 \right)(\C+\Hv)(\Vnum^n,v) + \alpha_1 (\C+\Hv)(\Vnum^{n,1},v)\Big],\\
  \ip{q^{n,2}}{w} =& \Kv(\Vnum^{n,2},w) ,
 \end{cases}
 \end{equation}
 and at the next time step $\tau^{n}+\delta$
\begin{equation}
\begin{cases} \label{eq:IMEXO3_stage3}
\ip{\Vnum^{n,3}}{v} =& \ip{u^{n}}{v}  +\delta \Big[\beta_1 \D(q^{n,1},\Vnum) + \beta_2 \D(q^{n,2},\Vnum) + \gamma \D(q^{n,3},\Vnum) \\
&+(1-\alpha_2) (\C+\Hv)(\Vnum^{n,1},v) + \alpha_2 (\C+\Hv)(\Vnum^{n,2},v)\Big],  \\
\ip{q^{n,3}}{w} =& \Kv(\Vnum^{n,3},w).
\end{cases}
\end{equation}
Finally 
\begin{align}
    \ip{\Vnum^{n+1}}{v} = \ip{u^{n}}{v}  +\delta & \Big[  \beta_1 \D(q^{n,1},\Vnum) + \beta_2 \D(q^{n,2},\Vnum) + \gamma \D(q^{n,3},\Vnum) \label{eq:IMEXO3_Final}\\
 &+ \beta_1(\C+\Hv)(\Vnum^{n,1},v) + \beta_2 (\C+\Hv)(\Vnum^{n,2},v) + \gamma  (\C+\Hv)(\Vnum^{n,3},v)\Big], \nonumber
\end{align}
and hence $\ip{q^{n+1}}{w}=\Kv(\Vnum^{n+1},w)$.

In order to ensure the stability of the method and measure its order of convergence we imposed a Courant–Friedrichs–Lewy (CFL) condition of the form
\begin{equation}
\label{eq:cfl}
\delta \leq  \frac{C \secondReviewer{h}}{(2k+1)\,|\sigma^2-\beta|\, \bar{S}},
\end{equation}
where $C \leq 1$ is a positive constant and $|\sigma^2-\beta|\bar{S}=\max_{S\in [0,\bar{S}]}|(\partial_{\V}f(S,\V)|$. Note that this condition implies that for a given level of refinement in the spatial mesh we must consider a small enough time step. \secondReviewer{Stability analysis for LDG-IMEX methods has been addressed in \cite{wang2015stability,wang2016stability}, for example.}

\section{Numerical results} \label{sec:numResults}
The results here presented are related to problems contained in Subsection \ref{subsec:close-out} and thus reformulated in Section \ref{sec:PDE and BSDE}.  We consider a European Call and a Put option, whose price is adjusted with an XVA, including KVA. \secondReviewer{Unless otherwise indicated,} the chosen parameters to define the model are shown in Table \ref{tab:param}.
In all numerical examples, \secondReviewer{as we consider the strike $K=15$}, the results have been obtained with $\overline{S}=60$ for domain truncation, \secondReviewer{i.e., $\overline{S}=4K$}.
\begin{table}[h!]
\centering
\begin{tabular}{lllllll}
\hline
$T=1$ year & $\sigma=0.3$ & $r=0.06$ & $\gamma_S=0$ & $q_S=0.06$ &  $R^B=0.7$ & $R^C=0.78$ \\
$\lambda^B=0.00133$ & $\lambda^C=0.0103$ & $\gamma_X=0.9$ & $r^X=0.07$ & $\gamma^\K=0.15$ & $\varphi=1$ & $\eta=0.08$ \\ 
LR $=0.03$ & $\omega= 0.75$ & $\alpha=1.4$ & $\SF=0.32$ &  $\sigma_r=1.5$ & $\text{RW}=0.05$ & \secondReviewer{$K=15$}\\
\hline
\end{tabular}
\caption{Value of all relevant parameters of the KVA-model.}
\label{tab:param}
\end{table}

\secondReviewer{In all numerical tests we have considered the time step equal to
\begin{equation}
\label{eq:deltal}
\delta =  \frac{0.5 h}{(2k+1)\,|\sigma^2-\beta|\, \bar{S}},
\end{equation}
thus satisfying the CFL condition \eqref{eq:cfl} for $C=0.5$. Note that the value of $\delta$ implies the one of $L$.
}

\secondReviewer{Plots of the XVA are shown in Figure \ref{fig:XVAs}, while the delta ($\Delta$) and the gamma ($\Gamma$) of the risky derivative are presented in Figures \ref{fig:deltas} and \ref{fig:gammas}, respectively. Both linear and nonlinear cases are presented}. \secondReviewer{These plots have been obtained by solving the system \eqref{eq:system_PDE} with the second-order LDG-IMEX \secondReviewer{($k=1$, IMEX scheme \eqref{eq:IMEXO2_stage1}-\eqref{eq:IMEXO2_stage2})} with $N=640$ and $L=115$. The value of the XVA is then obtained as the difference between the risky and risk-free derivative.} As expected, the total value adjustment is negative, and its module increases with the moneyness of the product. The numerical computations of the \secondReviewer{Greeks of the options, $\Delta$ and $\Gamma$}, show that our results are free from spurious numerical oscillations, a crucial point in hedging. 

Table \ref{tab:err and eoc} shows empirical errors and orders of convergence in $\Leb^2$ and $\Leb^\infty$ norms \secondReviewer{computed in the whole spatial domain} for different \secondReviewer{time and spatial }meshes, by using the reference solution computed with an enough refined mesh with $N=1280$ \secondReviewer{and $L=230$}. These results have been obtained with the second-order LDG-IMEX scheme \secondReviewer{($k=1$, IMEX scheme \eqref{eq:IMEXO2_stage1}-\eqref{eq:IMEXO2_stage2})}, for which we recover the theoretical order of convergence. 

A further check of the goodness of our results is done in 
Table \ref{tab:pdf vs fbsde}. Here we compare the results obtained by solving \secondReviewer{\eqref{eq:system_PDE}} with the ones obtained by solving the corresponding FBSDE. This latter equivalent formulation is solved through the stratified regression Monte-Carlo algorithm proposed in \cite{gltc2016}, which is based on least-squares Monte-Carlo and is very well suited for parallel computing. As expected by the uniqueness results of Section \ref{sec:PDE and BSDE} both methods converge to the same solution. Concerning the difference between the linear and the semilinear model in the value of the adjustment, it has the order of some basis points. From Table \ref{tab:pdf vs fbsde} and  Figure \ref{fig:XVAs}, we conclude that this difference increases with moneyness.

For the third-order scheme \secondReviewer{($k=2$, IMEX scheme \eqref{eq:IMEXO3_stage1}-\eqref{eq:IMEXO3_Final})}, Table \ref{tab:put order3} contains results for the nonlinear system \secondReviewer{\eqref{eq:system_PDE}} of a put option with the data in Table \ref{tab:param}. \secondReviewer{In this table we present results not only for the same tests as before but also for a convection-dominated setting (same market data presented in Table \ref{tab:param} although with lower volatility, set to $\sigma=0.05$). Note that the expected order of convergence is achieved.}

\begin{table}
\begin{center}
    \begin{tabular}{|c|c|c|c|c|c||c|c|c|c|}
    \hline
     $N$ & \secondReviewer{$L$} & $\Leb^2$-err & EOC & $\Leb^\infty$-err & EOC & $\Leb^2$-err & EOC & $\Leb^\infty$-err & EOC \\
    \hline
    \multicolumn{6}{|c||}{Put-linear} & \multicolumn{4}{|c|}{Put-nonlinear}  \\
\hline

10 & 1 & 5.544e-01 & - & 2.694e-01 & - & 5.541e-01 & - & 2.692e-01 & - \\
20 & 3 & 1.096e-01 & 2.339 & 4.260e-02 & 2.661 & 1.096e-01 & 2.338 & 4.260e-02 & 2.659 \\
40 & 7 & 2.787e-02 & 1.975 & 1.070e-02 & 1.993 & 2.787e-02 & 1.975 & 1.070e-02 & 1.993 \\
80 & 14 & 7.036e-03 & 1.986 & 2.718e-03 & 1.978 & 7.036e-03 & 1.986 & 2.718e-03 & 1.978 \\
160 & 28 & 1.792e-03 & 1.973 & 6.939e-04 & 1.97\secondReviewer{0} & 1.792e-03 & 1.973 & 6.939e-04 & 1.97\secondReviewer{0} \\
320 & 57 & 4.783e-04 & 1.905 & 1.903e-04 & 1.866 & 4.783e-04 & 1.905 & 1.903e-04 & 1.866 \\
640 & 115 & 1.055e-04 & 2.18\secondReviewer{0} & 4.837e-05 & 1.976 & 1.055e-04 & 2.18\secondReviewer{0} & 4.837e-05 & 1.976 \\
     \hline
 \multicolumn{6}{|c||}{Call-linear} & \multicolumn{4}{|c|}{Call-nonlinear}  \\
    \hline
10 & 1 &  1.592e+00 & - & 3.410e-01 & - & 1.602e+00 & - & 3.436e-01 & - \\
20 & 3 & 2.472e-01 & 2.687 & 4.888e-02 & 2.802 & 2.471e-01 & 2.697 & 4.887e-02 & 2.814 \\
40 & 7 & 5.354e-02 & 2.207 & 1.173e-02 & 2.059 & 5.344e-02 & 2.209 & 1.174e-02 & 2.058 \\
80 & 14 & 1.443e-02 & 1.891 & 3.000e-03 & 1.968 & 1.440e-02 & 1.892 & 3.001e-03 & 1.968 \\
160 & 28 & 3.995e-03 & 1.853 & 8.104e-04 & 1.888 & 3.986e-03 & 1.853 & 8.082e-04 & 1.893 \\
320 & 57 & 9.255e-04 & 2.11\secondReviewer{0} & 1.992e-04 & 2.025 & 9.236e-04 & 2.11\secondReviewer{0} & 1.992e-04 & 2.021 \\
640 & 115 & 1.576e-04 & 2.554 & 5.095e-05 & 1.967 & 1.573e-04 & 2.553 & 5.095e-05 & 1.967 \\
\hline
\end{tabular}
\end{center}
\caption{Error and empirical order of convergence (EOC) of the second order LDG-IMEX scheme. Results for $\Leb^2$ and $\Leb^\infty$ norms are presented for both linear PDE (on the left) and nonlinear PDE (on the right). The considered products are European Call and Put options solving \secondReviewer{\eqref{eq:system_PDE} with the parameters in Table \ref{tab:param}}. The solution obtained with $1280$ cells in space and $230$ time steps is taken as the reference solution.}
\label{tab:err and eoc}
\end{table}

\begin{table}
\centering
\begin{tabular}{|c|c|c|c|c|c|c|c|c|}
\hline
 & \multicolumn{2}{c}{Put-linear} & \multicolumn{2}{c|}{Put-nonlinear} &\multicolumn{2}{c}{Call-linear} & \multicolumn{2}{c|}{Call-nonlinear}\\
 \hline
$S$ & FBSDE & PDE &  FBSDE & PDE & FBSDE & PDE &  FBSDE & PDE \\
 \hline

 $5$  & -1.286e-01  & -1.266e-01 & -1.282e-01  &  -1.260e-01& -2.593e-02 & -2.557e-02 & -2.540e-02 & -2.555e-02\\
 $10$ &-5.070e-02  & -5.004e-02 &-5.126e-02  & -5.000e-02& -1.124e-02 & -1.127e-01& -1.121e-01& -1.123e-01 \\
 $15$ & -1.407e-02  &-1.395e-02 &-1.428e-02 & -1.395e-02& -2.598e-01& -2.624e-01  & -2.593e-01 &-2.615e-01 \\
 $20$ & -3.105e-03& -3.016e-03 &-3.117e-03  & -3.017e-03 & -4.517e-01&-4.571e-01 & -4.508e-01 & -4.555e-01\\
 $30$ &-1.159e-04  & -1.134e-04 & -1.177e-04 &-1.134e-04 & -8.680e-01 &-8.774e-01 & -8.673e-01& -8.742e-01\\
 $60$ & -7.540e-09  & -3.066e-09 & -7.517e-09 &-3.068e-09 & -2.101\secondReviewer{e-00} &-2.101\secondReviewer{e-00} & -2.094\secondReviewer{e-00} &-2.093\secondReviewer{e-00}\\
\hline
\end{tabular}
 \caption{For a given value of $S$ this table shows the value of the XVA at $t=0$ by solving the PDE or the FBSDE \secondReviewer{for the risky derivative with the data in Table \ref{tab:param}}. PDE values are given by the second-order LDG-IMEX \secondReviewer{ with $1280$ spatial cells and $230$ time steps}. Setup of Stratified Monte Carlo algorithm: piecewise linear approximations on $500$ cells of the spatial domain $[e^{-5},e^5]$. Besides, $10$ thousand simulations per cell and $20$ time steps were employed. }
\label{tab:pdf vs fbsde}
\end{table}
\begin{figure}
\centering
\includegraphics{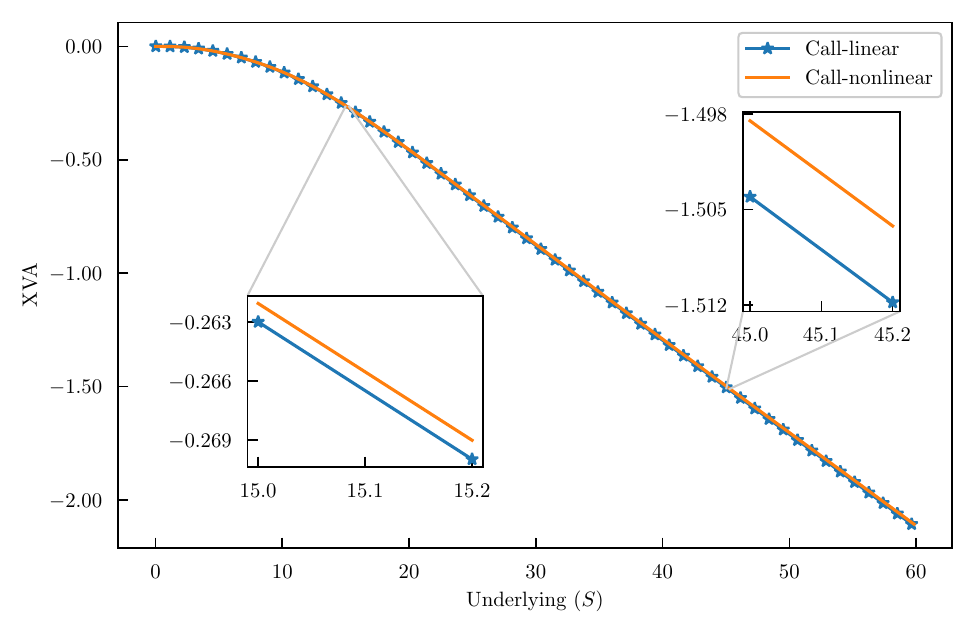}
\includegraphics{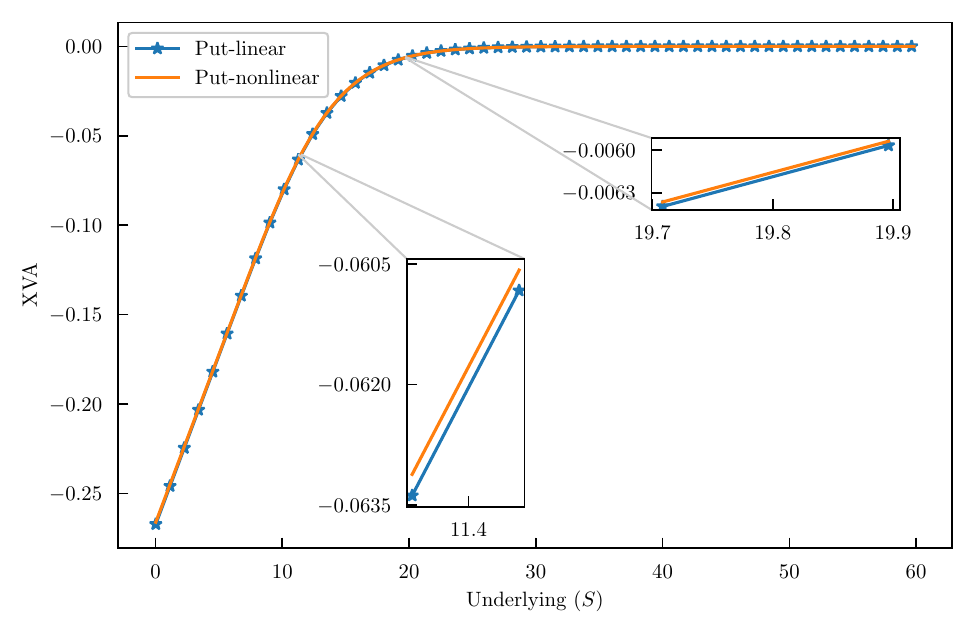}
\caption{Plot of the XVA at time $t=0$ for Call (above) and Put (below) European options with the data in Table \ref{tab:param}. Both the linear and the nonlinear cases are considered. }
\label{fig:XVAs}
\end{figure}

\begin{figure}
\centering
\includegraphics{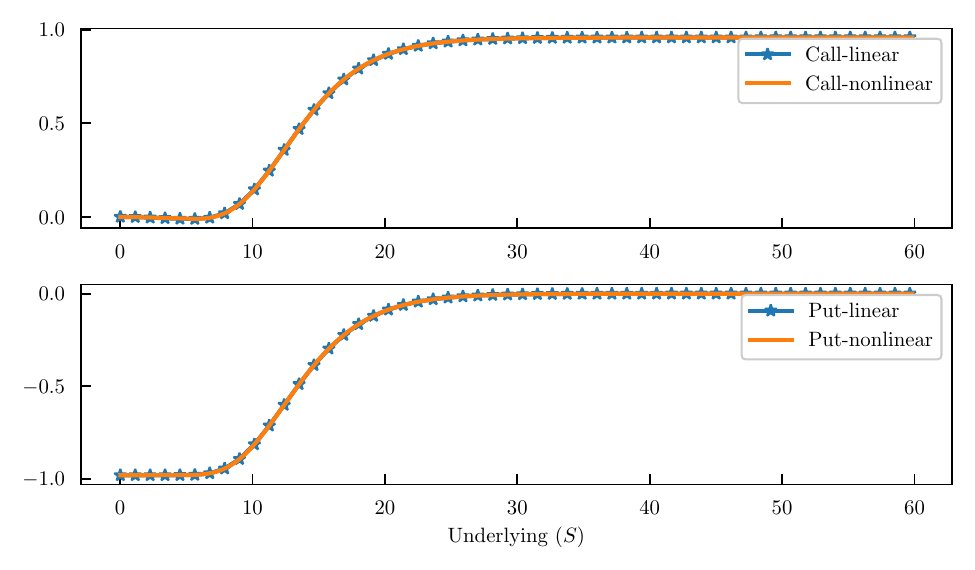}
\caption{Plot of $\Delta$ at time $t=0$ for Call (above) and Put (below) European options with the data in Table \ref{tab:param}. Both the linear and the nonlinear cases are considered. }
\label{fig:deltas}
\end{figure}

\begin{figure}
\centering
\includegraphics{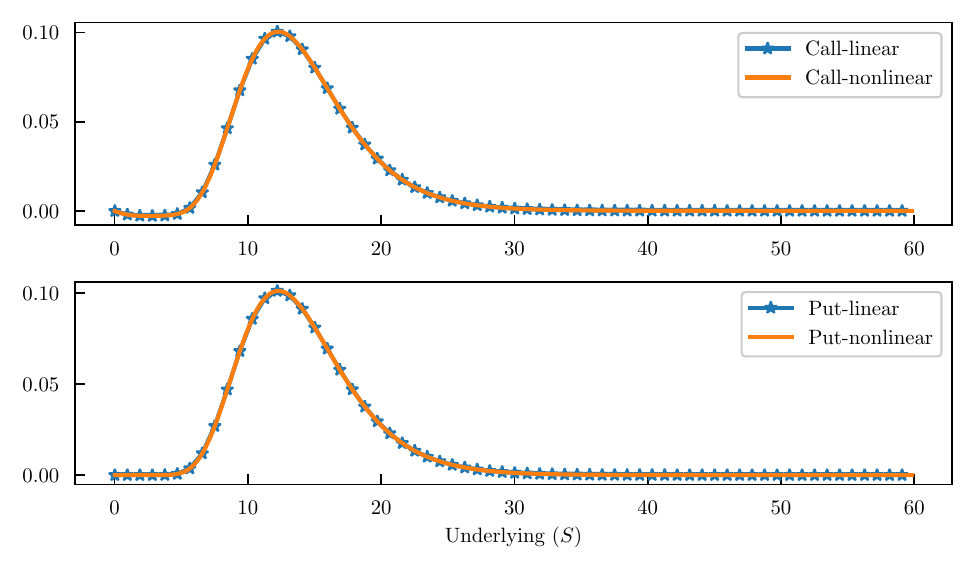}
\caption{\secondReviewer{Plot of $\Gamma$ at time $t=0$ for Call (above) and Put (below) European options with the data in Table \ref{tab:param}. Both the linear and the nonlinear cases are considered. }}
\label{fig:gammas}
\end{figure}
\begin{table}
\centering
\begin{tabular}{|c|c|c|c|c|c|c|c|c|c|c|}

    \hline
    & \multicolumn{5}{|c|}{$\sigma=0.3$}  & \multicolumn{5}{|c|}{\secondReviewer{$\sigma=0.05$}} \\
    \hline
     $N$ & \secondReviewer{$L$} & $\Leb^2$-err & EOC & $\Leb^\infty$-err & EOC & \secondReviewer{$L$} & $\Leb^2$-err & EOC & $\Leb^\infty$-err & EOC\\
    \hline
10 & 2 & 1.606e-01 & - & 6.327e-02 & - & 5 & 3.330e-01 & - & 1.747e-01 & - \\
20 & 5 & 5.097e-03 & 4.977 & 2.052e-03 & 4.947 & 11 & 4.868e-02 & 2.774 & 3.686e-02 & 2.245\\
40 & 11 & 6.648e-04 & 2.939 & 3.435e-04 & 2.578 & 22 & 4.922e-03 & 3.306 & 4.234e-03 & 3.122 \\
80 & 23 & 8.501e-05 & 2.967 & 4.514e-05 & 2.928 & 45 & 9.619e-04 & 2.355 & 7.041e-04 & 2.588 \\
160 & 47 & 1.090e-05 & 2.963 & 5.250e-06 & 3.104  & 91 & 1.222e-04 & 2.976 & 1.053e-04 & 2.741 \\
320 & 95 & 1.444e-06 & 2.916 & 6.143e-07 & 3.095 & 183 & 1.545e-05 & 2.984 & 1.376e-05 & 2.936 \\
640 & 190 & 1.816e-07 & 2.991 & 7.698e-08 & 2.996  & 367 & 1.936e-06 & 2.997 & 1.726e-06 & 2.995 \\
\hline
\end{tabular}
\caption{Error and empirical order of convergence (EOC) of the third-order LDG-IMEX scheme. Results for $\Leb^2$ and $\Leb^\infty$ norms are presented for Put-nonlinear.  Reference solutions with $1280$ cells in space, and $383$ and $735$ time steps when $\sigma=0.3$ and $0.05$, respectively.}
\label{tab:put order3}
\end{table}


\firstReviewer{Next, we illustrate the sensitivity of the option prices and XVA with respect to different parameters. First, in Figures \ref{fig:deltas_sigma} and \ref{fig:gammas_sigma} we show the graphs of deltas and gammas of the options for different values of the volatility parameter, ranging from $5\%$ to $30\%$ ({\it ceteris paribus} the other parameters in Table \ref{tab:param}). These figures show the transition to a convection-dominated setting where larger slopes of both Greeks are observed without any spurious numerical oscillations.
Secondly, also maintaining the rest of the parameters as in Table \ref{tab:param}, in Figure \ref{fig:KVA_comparison} we represent the XVA for different values of the cost of capital $\gamma^\K$, ranging from $6\%$ to $25\%$. The case $\gamma^\K=6\%$ represents the absence of KVA. We observe that, as the cost of capital increases, the XVA becomes more negative and the adjustment is more severe, especially at the \textit{in the money} regions. 
Thirdly, in Figure \ref{fig:CRA_comparison} we show the XVA for different values of the collateral rate $r^{X}$, ranging from $6\%$ to $10\%$. The case $r^{X}=6\%$ represents no consideration of the CRA. 
Analogously to the previous case, the adjustment is more severe for larger values of the collateral rate, and its influence increases with the moneyness.}
\begin{figure}
\centering
\includegraphics{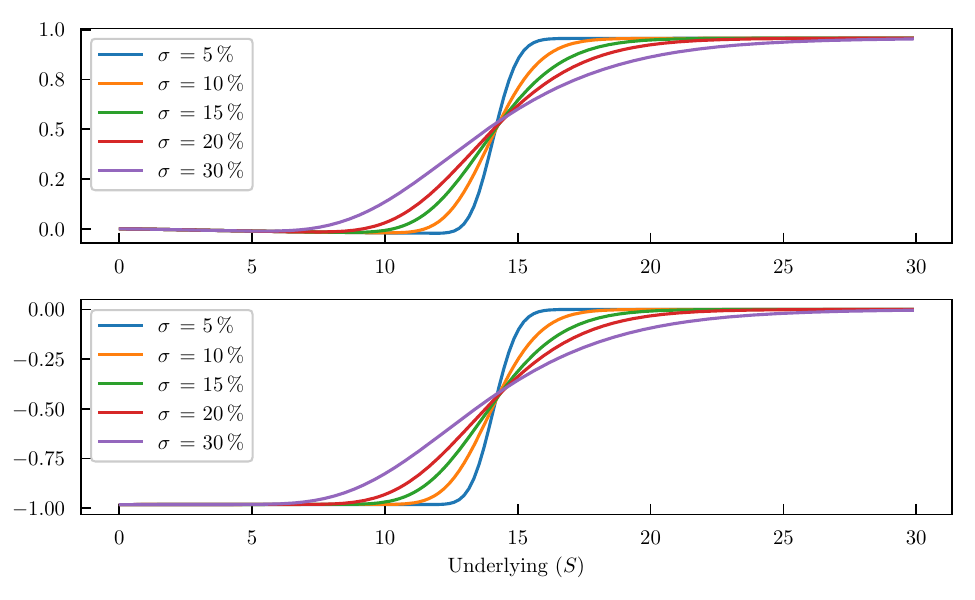}
\caption{Variation of $\Delta$ with respect to the stock volatility $\sigma$. Nonlinear PDE for the Call (above) and Put (below) options. } 
\label{fig:deltas_sigma}
\end{figure}
\begin{figure}
\centering
\includegraphics{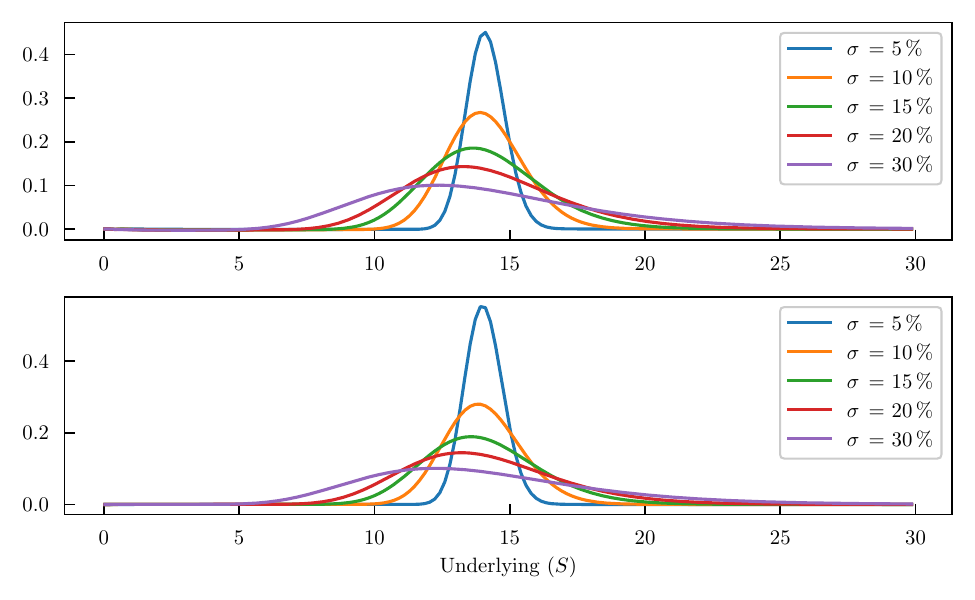}
\caption{Variation of $\Gamma$ with respect to the stock volatility $\sigma$. Nonlinear PDE for the Call (above) and Put (below) options. } 
\label{fig:gammas_sigma}
\end{figure}

\begin{figure}
\centering
\includegraphics{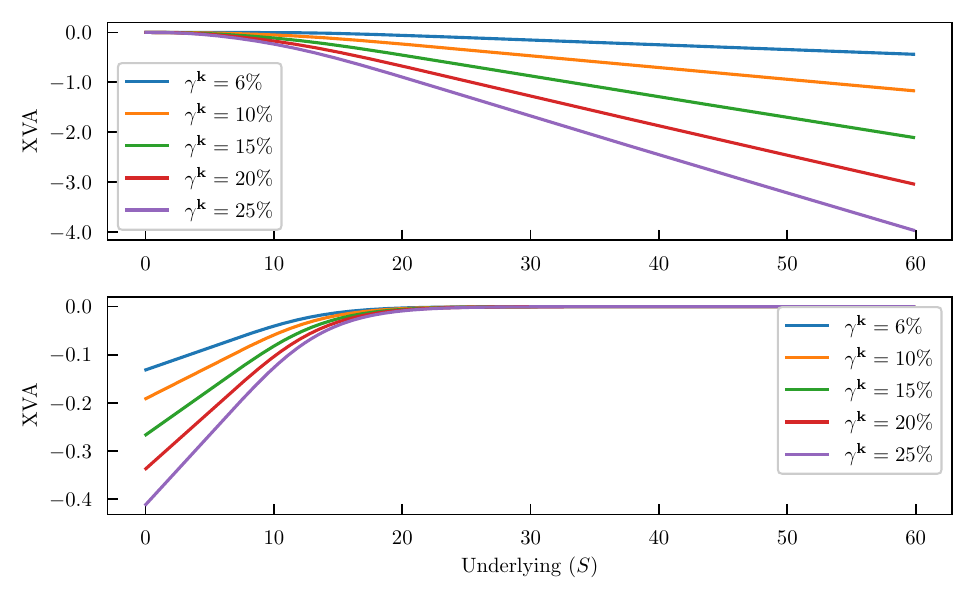}
\caption{Variation of XVA with respect to the cost of capital $\gamma^\K$. Nonlinear PDE for the Call (above) and Put (below) options.} 
\label{fig:KVA_comparison}
\end{figure}
\begin{figure}
\centering
\includegraphics{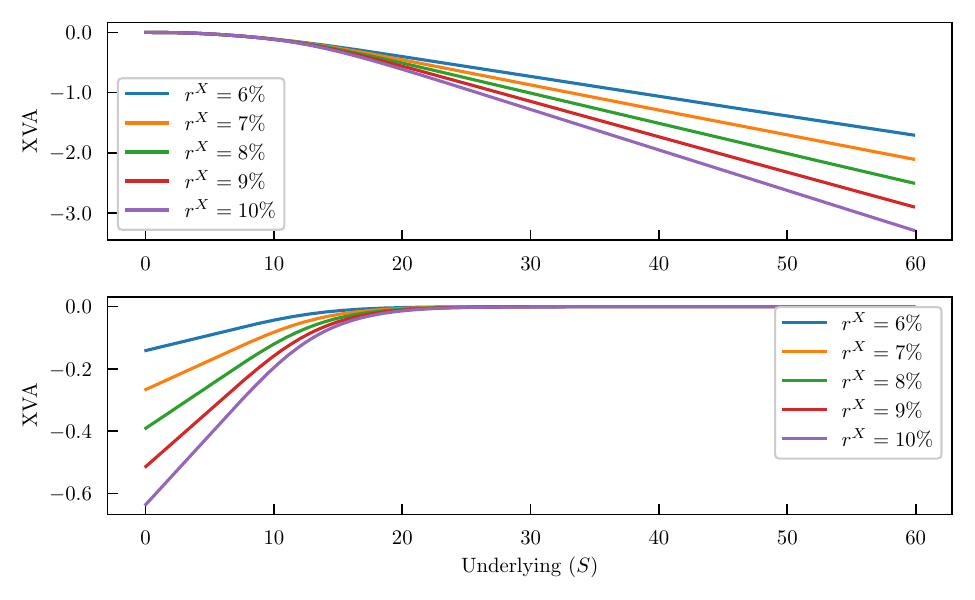}
\caption{Variation of XVA with respect to the collateral rate $r^X$. Nonlinear PDE for the Call (above) and Put (below) options.} 
\label{fig:CRA_comparison}
\end{figure}
\newpage

\appendix

\section{Admissible strategies and lack of arbitrage} 
\label{sec:lackofarb}
In Subsection \ref{subsec:asset-dividend} we introduced the couple $(Y,D)$ describing the prices and additional cashflows of the assets available in the economy. Then in Subsection \ref{subsec:semi-repl} we deduced a model to price a European derivative considering risks and costs which are not included in the classic Black-Scholes model. In the construction of a hedging portfolio $\theta$ we impose a single position in the collateral and capital account (see \eqref{eq:port-notation}). This condition for $\theta$ is valid in applications; in fact, every derivative can be associated with one collateral account and one capital requirement. In this part, we want to emphasize the relevance of this constraint from a theoretical market perspective. We claim that this constraint is necessary for the economy to be arbitrage-free.
In the first place, our model of economy contains a nonzero cumulative-dividend-process $D$, and a riskless account $B$. This latter can be used as a numeraire having zero diffusion. For this kind of economy, in line with \cite{duffie},
we give the following definition of Equivalent Martingale Measure (EMM).
\begin{definition}
\label{def:EMM}
Let $(Y,D)$ be a couple defined on $( \Omega, \F,\prob, (\F_t)_{t\in [0,T])})$ for which there exists a numeraire $B$ with zero diffusion. An equivalent martingale measure (EMM) of $(D,S)$ is a probability measure $\Q$ equivalent to $\prob$ such that
\begin{itemize}
\item $G^B=Y^B+D^B$ is a $\Q$-martingale,
where $Y^B_t:=\frac{Y_t}{B_t}$ and $dD^B_t=\frac{dD_t}{B_t}$. 
\item $\frac{d\Q }{d\prob}$ has finite variance.
\end{itemize}
\end{definition}
When $D=0$ this definition coincides with the one in the classical reference of \cite[Sec 2.1]{brigo2001interest}. 
\begin{remark}
For numeraires having diffusion, such as Zero-Coupon-Bond with diffusive short-rate, an additional term is needed in $dD^B$ in the above definition. This is necessary to guarantee a kind of unit invariance principle. 
\end{remark}
The assumption that an EMM measure exists is a standard practice in the financial industry. If the numeraire is bounded, such an assumption guarantees the absence of arbitrages in the space $\mathcal{H}^2(G)$ of square-integrable strategies.
For a complete definition of this space, we address the reader to \cite[Ch. 5-6]{duffie}. In particular, it contains strategies for which the stochastic integral with respect to the semi-martingale $G:=Y+D$ is well-defined, and portfolios are required to satisfy an integrability condition.
However, it is not hard to see that for the previously described economy, such an EMM measure cannot exist.
We notice that in the differential \eqref{eq:gain diff} the gain for the collateral and the capital account might grow at a rate that differs from the risk-free rate. 

By taking the collateral account $X$ as an example then
$$d\left(\dfrac{X_t}{B_t}\right)+ \dfrac{dD^X_t}{B_t}=\dfrac{dX_t}{B_t}-r_t\dfrac{X_t}{B_t}\, dt + r^X_t\frac{X_t}{B_t} - \frac{dX_t}{B_t}=(r_t-r^X_t)\frac{X_t}{B_t}\, dt,$$
which cannot be a martingale wrt any measure $\Q$, unless $r_t=r^X_t$. The same argument is valid for the capital account.
If there are accounts having different yields in the same economy, the existence of arbitrages becomes straightforward. To avoid this problem a restriction to the set of admissible strategies is then required. For $(Y,D)$ as in Subsection \ref{subsec:asset-dividend}, a trading strategy is an adapted process $\theta=(\theta_1,\theta_2,\ldots, \theta_7)$. Then the set of admissible strategies can be chosen as
$$ A:= \mathcal{H}^2(G)\cap \left\{\theta :\, \theta_1+\theta_2=\theta_1+\theta_3=0 \right\},$$
i.e., by imposing that each $\V$ is attached to one $X$ and $CA$.
\begin{prop}
Let $(Y, D)$ be as before with the component $\V$ in $Y$ solving the general XVA-adjusted price PDE \eqref{eq:general PDE price K} (whose existence is guaranteed by Corollary \ref{cor:existence classical sol}).
Then the economy of couple $(Y, D)$ and admissible strategies in $A$ is arbitrage-free.
\end{prop}
\begin{proof}
To the aim of this proof, we recall the definition of hedging error
\begin{equation*}
\epsilon_H:=\Delta^B \V + (1-R^B_t)(\V_t-X_t-CA)_t)
\end{equation*}
used in the dividend process related to $\V$.
In order to prove our statement we make use of an equivalent auxiliary couple $(Y',D')$. We define $Y'$ as $$Y':=(\V-X-CA, \REPO^S,\REPO^C, \PB, B)$$ 
and $D'$ analogously. The gain process is then given by $G'=Y'+D'$. 
We notice that for every $\theta=(\theta_1,\theta_2,\ldots, \theta_7)\in A$ we can associate $\theta'=(\theta_1,\theta_4,\theta_5,\theta_6,\theta_7)\in \mathcal{H}^2(G')$
such that
$$\theta \cdot Y= \theta' \cdot Y'.$$
Furthermore, the gain process generated by any $\theta\in A$ in $(Y, D)$ and the gain of the corresponding $\theta'$ in $(Y', D')$ are equal as well. As a result in $(Y,D)$ there are no arbitrages in $A$ if and only if in $(Y', D')$ there are no arbitrages in $\mathcal{H}^2(G)$.
The statement then follows by the Theorem in \cite[Ch. 6, Sec. F]{duffie} once we prove the existence of an EMM of $(Y', D').$ With regard to the component $\REPO^S$ of $Y$, the existence of an EMM $\Q$ is a well-known result under standard assumption on $\mu$ and $\sigma$. Under this measure, the dynamic of $S$ in the dividend process $dD^{\REPO^S}$ becomes
\begin{equation}
dS_t=(q_S-\gamma_S) S_t \,dt+ \sigma S_t \,dW^\Q_t,
\end{equation}
for a $\Q$-Brownian motion $W^\Q$.
The fact that $\Q$ is an EMM also on the components  $\REPO^C$ and $\PB$ follows immediately by \eqref{eq:spread-intensity}, while it is obvious for the riskless account $B$. It remains to verify that the gain $G^B$ is also a $\Q$-martingale in the component of $\V-X-CA$.
In this case, we have
\begin{equation}
\begin{aligned}
d(G'^B_1)_t&=d\left(\frac{(\V-X-CA)_t}{B_t}\right)+ \frac{d (D^{\V}-D^X-D^{CA})_t}{B_t}=
\\
&= \frac{1}{B_t}\left(d\V_t-dX_t-d(CA)_t-r_t (\V_t-X_t-(CA)_t)\, dt\right) \\
&+ \frac{1}{B_t}\left( -\epsilon^H_t dJ^{B}_t+dX_t- r^X_t X_t \, dt + d(CA)_t-\gamma^\K_t \K_t \, dt \right)\\
&=\frac{1}{B_t}\left( d\V_t-r^X_t X_t\, dt-\gamma^\K_t \K_t\, dt -\epsilon^H_t dJ^B_t-r_t (\V_t-X_t-\varphi \K_t)\, dt\right).
\end{aligned}
\end{equation} 
By Itô's Formula applied to jump-diffusion processes, we have
$$d\V_t=\left(\partial_t \V+\A^{(q_S-\gamma_S)}\V_t\right)\, dt +\Delta_C \V_t dJ^{C}_t +\Delta_B \V_t +\sigma S_t \frac{\partial \V}{\partial S} dW^\Q_t.$$
It stems from \eqref{eq:bound first der} that the last term $\sigma S_t \displaystyle \frac{\partial \V}{\partial S} dW^\Q_t$ constitutes a $\Q$-martingale. From \eqref{eq:general PDE price K}, after some simplifications we obtain
\begin{equation}
\begin{aligned}
d(G'^B_1)_t=&\frac{1}{B_t} \left( (r^B_t-r_t)(\V_t-X_t-\varphi \K_t) - \lambda^C_t \Delta_C\V_t \right)\, dt -\frac{1}{B_t}\left(\Delta_C \V_t \, dJ^C_t+ (\Delta^B \V_t-\epsilon_H)\, dJ^{B}_t \right).
\end{aligned}
\end{equation}
We observe that from \eqref{eq:def hedg err} and \eqref{eq:spread-intensity} we get
$$(r^B_t-r_t)(\V_t-X_t-\varphi \K_t)=\lambda^B_t(1-R^B_t)(\V_t-X_t-\varphi \K_t)=\lambda^B_t(\epsilon_H-\Delta_B\V_t),$$
and thus we have
\begin{equation}
\begin{aligned}
d(G'^B_1)_t&=\frac{1}{B_t} \left( \Delta_C\V( dJ^C_t-\lambda^C_t\, dt)+ \Delta_B\V_t (dJ^{B}_t-\lambda^B_t\, dt) -\epsilon^H_t(dJ^{B}_t-\lambda^B_t\, dt)\right).
\end{aligned}
\end{equation}
Since all jumps are compensated with their intensity then $G_1'^B$ is a $\Q$-martingale. This proves that $G'^B$ is a $\Q$-martingale and this concludes the proof.
\end{proof}
\section{Definition of Capital}
\label{sec:def of capital}
In this part, we give some regulatory generalities to compute capital requirements and apply them to obtain an explicit expression of $\K$ for a European Vanilla Option. This section does not cover all the possible cases contemplated by official regulation, so we readdress the reader to Basel documentation for further details. In general, capital requirement represents the minimum amount put into place by regulators to ensure that the firm does not take on excessive risk. The types of risk are contained in the decomposition of $\K$ itself. Specifically, the requirement that most
derivative trades are subject to is divided as follows
\begin{equation}
\label{eq:gen def of capital}
\K:=\max\left(\K_{MR}+\K_{CCR}+\K_{\CVA},\K_{LR}\right),
\end{equation}
where $\K_{MR}$, $\, \K_{CCR}$, $\, \K_{\CVA}$ and  $\K_{LR}$ are called respectively Market-risk, Counterparty credit risk, $\CVA$ and Leverage ratio capital. Each of these risks depends on the regulatory state of the firm and the counterparty's probability of default. Furthermore, the characteristics of the contract $\V$ are to be considered.
Except for $\K_{LR}$, these components are a percentage of the corresponding risk-weighted asset ($\RWA$). Namely,
\begin{equation*}
\begin{split}
&\K_{MR}=\eta \RWA_{MR},\\
&\K_{CCR}=\eta\RWA_{CCR},\\
&\K_{\CVA}=\eta \RWA_{\CVA},
\end{split}
\end{equation*}
where $\eta$ is the capital ratio. This latter represents the minimum percentage of TIER1 and TIER2 capital to be held by the institution. Its value is given by a supervisory authority but it is commonly set at $8\%$.
Market Risk Capital is a capital requirement held to offset the risk
of losses due to market risk. If the portfolio is completely hedged against this risk, it can be assumed to be zero. Therefore, in the following, we do not consider this capital component and set $\K_{MR}=0$. In particular, our regulatory capital does not take into account the FRTB capital contained in \cite[MAR 33]{marbasel}. We use the Standardised Approach for Counterparty Credit Risk (SACCR in \url{CRE52} on \url{bis.org}) to define a relevant measure called exposure-at-default measure (EAD). This latter is then used in the computation of the risk-weighted assets. Specifically, we also use the Standardized approach for counterparty credit risk capital to calculate $\RWA_{CCR}$, and we apply the Basic Approach to compute $ \RWA_{\CVA}$.
\subsection{EAD under Standardised method SACCR}
\label{subsect:EAD}
Under SACCR 
\begin{equation}
\EAD:=\alpha\times(\RC +\PFE).
\end{equation}
The replacement cost is given by $\RC:=(M-X)^+$,
and the potential future exposure is written as $\PFE:=m_t\times\AddOn_t$.
The multiplier $m_t$ is given by
\begin{equation}
    m_t:=\min\left\{ 1, 0.05+0.095\times\exp\left(\frac{M_t-X_t}{2\times 0.95\times \AddOn}\right)\right\},
\end{equation}
while $\AddOn=\text{SF} \times D$, where
SF denotes a constant supervisory factor (see \cite[CRE 52.72]{crebasel}), and $D$ is the effective notional $D:=d\times MF\times \delta$.
The adjusted notional $d$ is simply the value of the underlying, i.e. $d=S$. Assuming a day-count convention with $360$ business days, the maturity factor is given by
\begin{equation}
\label{eq:mat-fact}
 MF=\sqrt{\min((T-t)+ 10/360,1)}.
 \end{equation}
Finally, the supervisory delta for European Vanilla Options with strike $K$ is
\begin{equation}
\label{eq:def delta EAD}
 \delta=\begin{cases}
\phi\left(\frac{\log([S+0.01]/[K+0.01])+0.5 \sigma_r^2 (T-t)}{\sigma_r \sqrt{T-t}}\right) \quad \text{ Bought Call Option}, \\
-\phi\left(-\frac{\log([S+0.01]/[K+0.01])+0.5 \sigma^2 (T-t)}{\sigma_r \sqrt{T-t}}\right) \quad \text{ Bought Put Option},
\end{cases}
\end{equation}
where  $\phi$ denotes the cumulative distribution function of a standard normal random variable, and $\sigma_r$ is the supervisory volatility (see Table \cite[CRE 52.72]{crebasel}).
\subsection{Explicit capital formula}
We can now define the capital terms of \eqref{eq:gen def of capital}:
\begin{itemize}
\item[i)] The counterparty credit risk RWA is 
\begin{equation}
\RWA_{CCR}:=\omega\times 12.5 \times \EAD.
\end{equation}
The constant $\omega$ depends on the counterparty's rating and it can be found in \cite[CRE 20]{crebasel}.
\item[ii)] To the aim of $\RWA_{\CVA}$, we need to consider the maturity $M_t:=\min(1, (T-t))$, and the risk-weight RW of the counterparty (see Table 1 \cite[MAR 50.16]{marbasel}). Then
\begin{equation}\label{eq:CVA RWA}
\RWA_{\CVA}:= \frac{12.5\times 0.65}{\alpha}\times \text{RW}\times \text{M}_t \times \EAD\times \frac{1-e^{-0.05\times \text{M}_t}}{0.05\times \text{M}_t}.
\end{equation}
\item[ii)] Differently from other capital components, Leverage ratio capital is the Capital Measure 
\[\K_{LR}:=\text{Capital Measure}=\text{Exposure Measure}\times \text{LR}.\]
The Leverage Ratio (LR) is a percentage chosen by a regulator of at least $3\%$. In the concern of the Exposure Measure, in our case we have
\begin{equation}
\text{Exp. Measure}:=\max(M,0)+\AddOn,
\end{equation}
where  the AddOn is defined as in \ref{subsect:EAD}.
\end{itemize}
\section{Notation}
\label{sec:notation}
\begin{center}
    \begin{tabular}{|l|l|}
    \hline
    Parameter & Description\\
    \hline
        $\V$ & The value of the adjusted derivative \\
        $V^f$ & The value of the derivative considering CVA, DVA, and FVA\\
        $V$  & The risk-free value of the derivative\\
        $U$ & The value of the adjustment (XVA)\\
        $M$ & Mark-to-Market Value of the derivative (MTM)\\
        $S$ & Underlying stock \\
        $\mu$,  $\sigma$ & Stock growth and volatility\\
        $\gamma_S$ & \secondReviewer{Dividend} yield of the stock\\
        $q_S$& \secondReviewer{Repo} rate of $S$ \\
        $r$ & Risk-free rate\\
        $r^B$ & Funding rate on the \secondReviewer{issuer} bond\\
        $\lambda^B$ & Intensity of default of $\PB$\\
        $r^C$ & Yield on the counterparty bond\\
        $q^C$ & \secondReviewer{Repo} rate of $\PC$\\
        $\lambda^C$ & Intensity of default of $\PC$\\
        $\gB$ & Close-out in case of default of the \secondReviewer{issuer}\\
        $\gC$ & Close-out in case of default of the counterparty\\
         $X$ & Collateral account \\
        $r^X$ & Yield of the collateral account \\
        \hline
    \end{tabular} 
    \\
    \vspace{10 pt}
    \begin{tabular}{|l|l|}
    \hline 
    Parameter & Description\\
    \hline
        $\K$ & Capital requirement \\
        $\gamma^\K$ & Capital hurdle rate\\
        $\varphi\in [0,1]$ & Amount of $\K$ used to fund the position $\V$\\
        $\K_{MR}$& Market-risk capital\\
        $\K_{CCR}$& Counterparty-credit-risk capital\\
        $\K_{\CVA}$ & \CVA-capital\\
        $\K_{LR}$ & Leverage-ratio capital\\
        $\Eq$ & Equity financing the position $\V$\\
        $\RE$ & Retained earnings\\
        $\RWA$ & Risk-weighted-assets\\
        $\EAD$ & Exposure at default \\
        $\PFE$ & Potential Future Exposure\\
        $\RC$ & Replacement cost\\
        $\AddOn$ & Regulatory Add-On\\
        $\MF$ & Maturity factor \\
        $\SF$ & Supervisory factor\\
        \hline
    \end{tabular}
\label{tab:notat xva}
\end{center}

\section*{Acknowledgements}
This work has been partially funded by EU H2020-MSCA-ITN-2018 (ABC-EU-XVA Grant Agreement 813261). The authors also acknowledge the funding by the Spanish Mi\-nis\-try of Science and Innovation through the project PID2022-141058OB-I00 and from the Galician Government through grants ED431C 2022/47, both including FEDER financial support, as well as the support received from the Centro de Investigaci\'on en Tecnolog\'{\i}as de la Informaci\'on y las Comunicaciones de Galicia, CITIC, funded by Xunta de Galicia and the European Union (European Regional Development Fund, Galicia 2014-2020 Program) through the grant ED431G 2019/01.

Finally, the authors would like to thank the comments and work of two anonymous reviewers that helped to improve the article.

  \bibliographystyle{plain} 
  \bibliography{references}
\end{document}